\documentclass{amsart}
\usepackage{amsfonts}
\usepackage{amsmath}
\usepackage{amssymb}
\usepackage{graphicx}
\usepackage{amsthm}
\usepackage{graphicx}
\usepackage{calc}
\usepackage{float}
\usepackage{tikz}
\usetikzlibrary{decorations.markings}

\tikzstyle{vertex}=[circle, draw, inner sep=0pt, minimum size=6pt]
\newcommand{\vertex}{\node[vertex]}

\newtheorem{theorem}{Theorem}
\theoremstyle{plain}

\newtheorem{corollary}{Corollary}

\newtheorem{definition}{Definition}
\newtheorem{example}{Example}

\newtheorem{lemma}{Lemma}

\newtheorem{remark}{Remark}

\numberwithin{equation}{section}

\begin{document}

\title[Graphs with minimal well-covered dimension]{Graphs with minimal well-covered dimension}
\author{Gabriella Clemente}
\email{gclemen00@citymail.cuny.edu}
\address{The City College of New York, 166 Convent Avenue, New York, NY 10031}
\dedicatory{To my father, L. Clemente, my mother, V. Alonso, and to all women of mathematics.}

\begin{abstract}
There is a class of graphs with well-covered dimension equal to the simplicial clique number that contains all chordal graphs and infinitely many other graphs. These graphs generalize a result by Brown and Nowakowski on the well-covered dimension of chordal graphs. Furthermore, each member of the infinite family of Sierpinski gasket graphs of order at least $2$ has well-covered dimension $3,$ the simplicial clique number. 
\end{abstract}

\maketitle

\section{Introduction} 

We start by giving a few graph theoretical definitions, leading up to the concepts of well-covered weighting, space and dimension, which are the main object of study in this paper. We refer the reader to \cite{DW} for any notions we use but do not define, and for the more detailed treatment of related ideas. 

Throughout, we assume that graphs are simple, connected, and undirected. A graph is a set of vertices $V(G)$ and edges $E(G)$ with specified connectivity relations. We usually write $G=(V(G),E(G))$ for a graph and $vw$ for the edge connecting the vertices $v$ and $w$. Also, we say that the order of $G$ is $|V(G)|$ and that two vertices $v$ and $w$ are adjacent if there is an edge between them.  

Let $I \subseteq V(G)$. The neighborhood of $I$, denoted by $N(I)$, is the set of all vertices that are adjacent to any vertex in $I$. The closed  neighborhood of $I$ is $N[I]=N(I)\cup I$. When $I=\{v\}$, we write $N(v)$ and $N[v]$. The degree of a vertex $v$ is $|N(v)|$. A complete graph of order $n$, $K_n$, is such that the degree of each of its vertices is $n-1$.

A subgraph of a graph $G=(V_1,E_1)$ is a graph $H=(V_2,E_2)$ with $V_2 \subseteq V_1$ and $E_2 \subseteq E_1$. We say that $X \subseteq V(G)$ is a clique of $G$ if $X$ induces the subgraph $K_{|X|}$. A clique of $G$ is maximal if it is not properly contained in any other clique of $G$.  A graph is chordal if it has no induced cycles of length at least $4$. A more thorough introduction to these ideas can be found in \cite{RT}. 

For conciseness, we write $\mathbb{N}_{\leq a}$ for $[1, a]\cap \mathbb{N}$, and $\mathbb{N}_{\geq a}$ for $[a, \infty) \cap \mathbb{N}$.

\vspace{.1in}

Our results rely heavily on the definitions that follow. 

\begin{definition}
Let $G$ be a graph and $L \subset V(G)$. The set $L$ is independent if no two vertices in $L$ are adjacent. If $L$ is not properly contained in any independent set of $G$, then we say that $L$ is a maximal independent set (\emph{MIS}) of $G$. 
\end{definition}

\begin{definition}
Let $\mathbb{F}$ be a field and $G$ be a graph.
\begin{enumerate}
\item A weighting of $G$ is a map $f:V(G) \rightarrow \mathbb{F}$. If a weighting $f$ is such that 
\[
\sum_{v \in \mathcal{M}} f(v)
\]
is constant for every \emph{MIS} $\mathcal{M}$ of $G$, then $f$ is said to be a well-covered weighting of $G$.
\item The vector space (over $\mathbb{F}$) of all well-covered weightings of $G$, which we designate with $\mathcal{V}$, is called the well-covered space of $G$. 
\item The well-covered dimension of $G$ over $\mathbb{F}$ is $wcdim(G,\mathbb{F}) = \dim_{\mathbb{F}}(\mathcal{V})$. 
\end{enumerate}
\end{definition}

The well-covered dimension of any graph is, clearly, at most its order, as remarked in \cite{CY}. 

\begin{remark}
There are graphs whose well-covered dimension depends on the characteristic of the field $\mathbb{F}$ (see \cite{BN} and \cite{BK}). The graphs we study in this article have well-covered dimension independent of field characteristic. Hence, we omit explicit reference to $\mathbb{F}$ and write $wcdim(G)$ instead of $wcdim(G,\mathbb{F})$. 
\end{remark}

\begin{definition}
Let $G$ be a graph. 
\begin{enumerate}
\item $v\in V(G)$ is a simplicial vertex of $G$ if $N[v]$ is a maximal clique. 
\item A clique of $G$ is simplicial if it contains at least one simplicial vertex. 
\item $\mathcal{C}(G)$ is the set of all simplicial cliques of $G$ and $sc(G): = |\mathcal{C}(G)|$. We say that $sc(G)$ is the simplicial clique number of $G$ and we denote the $i$-th member of $\mathcal{C}(G)$ by $C_i$.
\item A clique covering of $G$ is a family of cliques whose union is $V(G)$.
\end{enumerate}
\end{definition} 

\vspace{.1in}

The notions of well-covered weighting and well-covered space of a graph originate from the concept of well-coveredness of a graph. Well-covered graphs were first studied by Plummer (see \cite{P} and \cite{PP}). These graphs have the property that all of their \emph{MIS}s have equal cardinality. The notion of well-coveredness of graphs may be stated in terms of weights of vertices, as noted in \cite{CER}. Let $f:V(G) \rightarrow \mathbb{Z}$ be such that $f(v)=1$, for all $v \in V(G)$. Then, $G$ is well-covered if for every \emph{MIS} $\mathcal{M}$ of $G$, $\sum_{v \in \mathcal{M}} f(v) = |\mathcal{M}|$ is constant. 

It is now possible to ask the following central question: given a graph $G$, what properties need a weighting of $G$, $f$, have in order for $\sum_{v \in \mathcal{M}} f(v)$ to be constant, for any \emph{MIS} $\mathcal{M}$ of $G$? In other words, what do the well-covered weightings of $G$ look like? This problem was first posed in \cite{CER}, where the authors observed that the well-covered space of any graph is non-empty, for the zero-function is a trivial well-covered weighting of any graph. But they also observed that only some graphs have well-covered weightings other than the zero-function. Examples of graphs with unique well-covered weighting the zero-function are cycles of length al least 8, which are studied in detail in \cite{BK}. Thus, a graph that is not well-covered \emph{can be made well-covered in terms of its vertex weights} via a well-covered weighting. 

The notions of well-covered space and well-covered dimension have been studied in the more general setting of hypergraphs $H$, and weightings whose domain need not be $V(H)$ (see \cite{CY}). 

Brown and Nowakowski \cite{BN} proved that for any graph $G$, $wcdim(G)\geq sc(G)$, and that equality holds if $G$ is chordal. Thus, chordal graphs have minimal well-covered dimension. In the following sections, we show that there is a large family of graphs with minimal well-covered dimension (just like chordal graphs). This family contains the class of  chordal graphs properly, and thus generalizes Brown and Nowakowski's result. There are several open questions in the well-covered dimension theory of graphs. Some of these questions are the following:

\begin{enumerate}
\item What exactly does the well-covered dimension of a graph tell us about a graph? 
\item How exactly are clique partitions and coverings, and the well-covered dimension of a graph related? 
\item Find a non-trivial upper bound for the well-covered dimension of any graph; that is, an upper bound other than the order of a graph.
\item Classify all graphs according to their well-covered dimension. Suggestion: start by finding the largest possible class of graphs to which our results apply.
\end{enumerate}

The remainder of this paper is organized as follows. In section $(2),$ we define the class of simplicial clique covered graphs, the \emph{SCCG} class, and prove that any graph in this class has well-covered dimension the number of simplicial cliques. In section $(3),$ we introduce the class of simplicial clique sums of a chordal graph with a \emph{SCCG}, which results in a generalization of our main theorem in section $(1)$ and Brown and Nowakowski's theorem on the well-covered dimension of chordal graphs. In section $(4),$ we prove that all Sierpinski gasket graphs of order at least $2,$ have well-covered dimension $3,$ which happens to be the simplicial clique number. This suggests that the main theorem in section $(3)$ could be generalized further. A full generalization of this theorem would be a big step in answering open question $(4),$ and would advance our overall  understanding of the well-covered dimension theory of graphs.

\section{The well-covered dimension of simplicial clique covered graphs}
 
In this section we investigate a class of graphs that overlaps, but is not identical to, the class of chordal graphs. Our goal is to prove that the well-covered dimension of a graph in this class is equal to its simplicial clique number. From now on, all well-covered weightings are assumed to be non-trivial.

\begin{definition}\label{defSCCG}
A graph $G$ is a simplicial clique covered graph (\emph{SCCG}) if $\mathcal{C}(G) \neq \emptyset$ and $\mathcal{C}(G)$ is a clique covering of $G$.
\end{definition} 

We now we present some technical definitions and notation.

\begin{definition}
Let \(G\) be a \emph{SCCG}. 
\begin{enumerate}
\item The connection set, $\mathcal{W}$, of $G$ is the set of vertices that belong to at least two simplicial cliques of $G$.
\item For an independent set $I_m \subset \mathcal{W},$ $S(I_m)$ is the set of all simplicial cliques that are not contained in the closed neighborhood of $I_m$ and $s_m:=|S(I_m)|.$
\item $\overline{S}(I_m)$ is the complement of $S(I_m)\) in \(\mathcal{C}(G)$.
\item For each $C_i\in \mathcal{C}(G)$, $W_i:=C_i\cap \mathcal{W}.$ 
\end{enumerate}
\end{definition}

\begin{remark}\label{remS_mempty}
From the previous definition, it is evident that
\begin{enumerate}
\item $I_m$ is a \emph{MIS} of $G$ if and only if $S(I_m)=\emptyset$.
\item 
\[
\mathcal{W}=\bigcup_{i=1}^{sc(G)} W_i.
\]
\end{enumerate}
\end{remark}

\hspace{.1in}

Our first result is a classification of the \emph{MIS}s of any \emph{SCCG}. We shall soon see that \emph{MIS}s reveal much about well-covered spaces and their dimension.

\begin{theorem}\label{miscgthm}
Let $G$ be a \emph{SCCG} and $\mathcal{M}$ be a \emph{MIS} of $G$. Then, either
\begin{enumerate}
\item $\mathcal{M}=\{v_1,\hdots,v_{sc(G)}\}$, where each $v_i \in \mathcal{M}$ is a simplicial vertex in a distinct $C_i \in \mathcal{C}(G)$, or  
\item $\mathcal{M}=I_m \cup \mathcal{M}'$, where $I_m$ is an independent set of $\mathcal{W}$ and $\mathcal{M}'$ consists of one simplicial vertex per $C_i \in S(I_m)$.
\end{enumerate} 
\end{theorem}

\begin{proof}
Assume that $\mathcal{M}$ is not as in (1).  Then, $\mathcal{M} \cap \mathcal{W} \neq \emptyset$, and thus $\mathcal{M} \cap \mathcal{W}=I_m,$ for some independent set $I_m$ of $\mathcal{W}.$ 

If $\mathcal{M}-I_m= \emptyset$, then $I_m$ is a \emph{MIS} of $G,$ and thus (see Remark \ref{remS_mempty}) $\mathcal{M}$ is as in (2) with $\mathcal{M}'=\emptyset$.

If $\mathcal{M}-I_m=\mathcal{M}' \neq \emptyset$, then it is clear that $\mathcal{W} \cap \mathcal{M}' =\emptyset$. It follows that each $v_i \in \mathcal{M}'$ must be simplicial and  non-adjacent to vertices in $I_m$. Thus, $\mathcal{M}$ is as in (2).
\end{proof}

With all notation in place and with the help of Theorem \ref{miscgthm}, it is possible to count the \emph{MIS}s of any \emph{SCCG}.

\begin{theorem}\label{countmisthm}
Let $G$ be a \emph{SCCG}. Then, $G$ has exactly 
\[
|\mathcal{I}|+ \prod_{i=1}^{sc(G)} |C_i -W_i|+\sum_{m=1}^M \prod_{C_i \in S(I_m)} |C_i - W_i|
\]
maximal independent sets, where $\mathcal{I}$ is the family of all independent sets of $\mathcal{W}$ that are \emph{MIS}s of $G.$
\end{theorem}

\begin{proof}
By Theorem \ref{miscgthm}, each \emph{MIS} of $G$ takes one of two forms. Let $\mathcal{M}$ be a \emph{MIS} of $G$.

Suppose that $\mathcal{M}$ is of form (1) in Theorem \ref{miscgthm}. Then, each $v_i\in \mathcal{M}$ is exactly one out of the $|C_i - W_i|$ simplicial vertices of $C_i \in \mathcal{C}(G)$. Hence, there are 
\[
\prod_{i=1}^{sc(G)} |C_i -W_i|
\]
\emph{MIS}s of this form.

Suppose that $\mathcal{M}$ is of form (2) in Theorem \ref{miscgthm} and that $G$ has exactly $M$ independent sets $I_m$. Since each $v_i \in \mathcal{M} - I_m$ can be exactly one out of the $|C_i - W_i|$ simplicial vertices of $C_i \in S(I_m)$, there are 
\[
\sum_{m=1}^M \ \prod_{C_i \in S(I_m)} |C_i - W_i|
\]
\emph{MIS}s of this form.

Observe that 
\[
\prod_{C_i \in S(I_m)} |C_i - W_i|
\]
vanishes for those $I_m$ that are \emph{MIS}s of $G$ because $S(I_m)=\emptyset$. Thus, we must add all independent sets of $\mathcal{W}$ that are \emph{MIS}s of $G.$ Letting $\mathcal{I}$ be the family of all such sets, we add $|\mathcal{I}|$ \emph{MIS}s to complete our count.
\end{proof}

Next, we look at the defining properties of well-covered weightings of \emph{SCCG}s.

\begin{lemma}\label{lemmafconstant}
Let $G$ be a \emph{SCCG}. Let $\mathcal{W}$ be a connection set and $f$ be a well-covered weighting of $G$. Then, $f$ is constant on $C_i-W_i$ for each $C_i \in \mathcal{C}(G)$.
\end{lemma}

\begin{proof}
Let $\mathbb{F}$ be a field and let $f:V(G) \rightarrow \mathbb{F}$ be a well-covered weighting of $G$.

Pick an arbitrary $C_i \in \mathcal{C}(G)$. By (1) of Theorem \ref{miscgthm} together with Theorem \ref{countmisthm}, we can find $|C_i-W_i|$ \emph{MIS}s of $G$ of cardinality $sc(G)$ that have $sc(G)-1$ vertices in common and as the $sc(G)$-th vertex, a distinct $v_i \in C_i-W_i$. Then, all of the vertices in $C_i - W_i$ have the same weight under $f$.  Since $C_i$ was chosen arbitrarily, the result follows. 
\end{proof}

\begin{lemma}\label{lemmafsum}
Let $G$ be a \emph{SCCG}. Let $\mathcal{W}$ be a connection set and $f$ be a well-covered weighting of $G$. For any $w \in \mathcal{W}$, 
\[
f(w)=\sum f(v),
\]
where the sum is taken over a set of simplicial vertices, each of which belongs to a distinct $C_i \in \overline{S}(\{w\})$.
\end{lemma}

\begin{proof}
Let $w \in \mathcal{W}$ and let $I_m=\{w\}$. For any $i \in \mathbb{N}_{\leq sc(G)}$, let $u_i$ and $v_i$ be simplicial vertices of $G$ such that each $u_i$ belongs to a distinct $C_i \in S(I_m)$ and each $v_i$ belongs to a distinct $C_i \in \overline{S}(I_m)$. Consider a set  $\mathcal{M}=I_m \cup \{u_1, \hdots, u_{s_m}\}$ and a set $\{v_1,\hdots,v_{sc(G)-s_m}\}$. Note that $\mathcal{M}$ is of form (2) in Theorem \ref{miscgthm} and $\mathcal{M}'=(\mathcal{M}-I_m)\cup \{v_1,\hdots,v_{sc(G)-s_m}\}$ is of form (1) in Theorem \ref{miscgthm}. It follows that 
\[
f(w)=\sum f(v),
\]
where each $v$ in the sum is simplicial and belongs to a distinct $C_i \in \overline{S}(\{w\})$.
\end{proof}

\begin{remark}
When needed, we use the following notation for vectors in $\mathbb{F}^{n}:$
\[
(a_1^{n_1} \mid a_2^{n_2} \mid \hdots \mid a_k^{n_k}):=(\underbrace {a_1,\hdots,a_1}_{n_1-times},\underbrace{a_2,\hdots,a_2}_{n_2-times},\hdots,\underbrace{a_k,\hdots,a_k}_{n_k-times}),
\]
where $n=\sum_{i=1}^k n_i$.
\end{remark}

Let $G$ be a \emph{SCCG}. We identify each well-covered weighting of $G$ with an $n$-tuple $\textbf{x}=(f(v_1),\hdots,f(v_n)) \in \mathbb{F}^{n}$. We call the vector space of all such $n$-tuples $\mathbb{V}$. It is clear that $wcdim(G) = \dim(\mathbb{V})$.

Let \(\mathcal{W}\) be a connection set of $G$. For any $i \in \mathbb{N}_{\leq sc(G)}$, let $t_i=|C_i - W_i|$ and let $l=|\mathcal{W}|$. Suppose that $G$ has order $n=k+l$, where $k=\sum_{i=1}^{sc(G)} t_i$. Then, using Lemma  \ref{lemmafconstant},  any vector $\textbf{x} \in \mathbb{V}$ may be expressed as 
\[
\textbf{x}  = (f(w_1),\hdots,f(w_l)\mid a_1^{t_1} \mid a_2^{t_2} \mid  \hdots \mid a_{sc(G)}^{t_{sc(G)}}),
\] 
where we have placed the weights of the connection vertices of $G$ first. Now we can use Lemma \ref{lemmafsum} to get that the first $l$ components of $\textbf{x}$ are linear combinations of the $a_i$'s. It follows that every $\textbf{x} \in \mathbb{V}$ can be written as a linear combination of at most $sc(G)$ linearly independent vectors. This means that $wcdim(G)\leq sc(G)$. Since we already knew that $wcdim(G)\geq sc(G)$, for any graph $G$,  we obtain the main theorem of this section, which is stated below. 

\begin{theorem}\label{thmGabriella}
Let $G$ be a \emph{SCCG}. Then, $wcdim(G)=sc(G)$.
\end{theorem}

We now give some examples of applications of Theorem \ref{thmGabriella}. 

\begin{example}
Consider the graph $G$ given by
\begin{center}
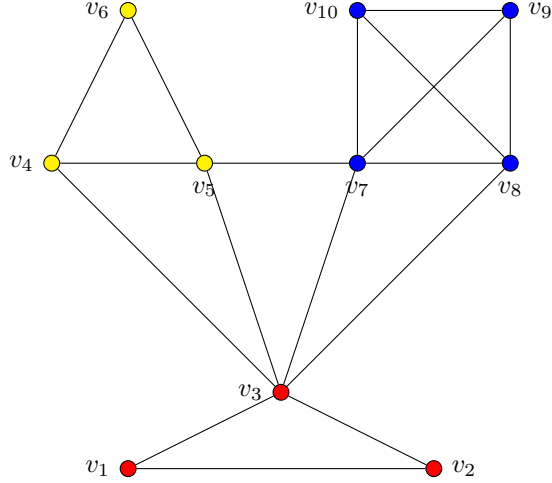
\begin{figure}[H]
\begin{tikzpicture}[x=.4in, y=.4in]
\vertex (w1) at (0,0) [label=below:$v_7$, style={fill=blue}]{};
\vertex (w2) at (2,0)[label=below:$v_8$, style={fill=blue}]{};
\vertex (w3) at (2,2)[label=right:$v_9$, style={fill=blue}]{};
\vertex (w4) at (0,2)[label=left:$v_{10}$, style={fill=blue}]{};
\vertex (w5) at (-2,0)[label=below:$v_5$, style={fill=yellow}]{};
\vertex (w6) at (-4,0)[label=left:$v_4$, style={fill=yellow}]{};
\vertex (w7) at (-3,2)[label=left:$v_6$, style={fill=yellow}]{};
\vertex (w8) at (-1,-3)[label=left:$v_3$, style={fill=red}]{};
\vertex (w9) at (-3,-4)[label=left:$v_1$, style={fill=red}]{};
\vertex (w10) at (1,-4)[label=right:$v_2$, style={fill=red}]{};
\path 
(w1) edge (w2)
(w2) edge (w3)
(w3) edge (w4)
(w4) edge (w1)
(w1) edge (w3)
(w2) edge (w4)
(w1) edge (w5)
(w5) edge (w6)
(w6) edge (w7)
(w5) edge (w7)
(w6) edge (w8)
(w5) edge (w8)
(w1) edge (w8)
(w2) edge (w8)
(w8) edge (w9)
(w9) edge (w10)
(w8) edge (w10);
\end{tikzpicture}
\caption{A \emph{SCCG} with $sc(G) = 3$ and empty connection set.}
\label{exampleSCCG1}
\end{figure}
\end{center}

$G$ is a \emph{SCCG} with $\mathcal{C}(G)=\{C_1,C_2,C_3\}$, where $|C_1|=|C_2|=3$ and $|C_3|=4$. For this particular $G$, $\mathcal{W}=\emptyset$. Note that $\mathcal{C}(G)$ is a minimum clique partition of $G$.

A basis for the well-covered space of $G$ is \[B_G=\{(1^3 \mid 0^7), (0^3 \mid 1^3 \mid 0^4), (0^6 \mid 1^4)\},\] using Lemma \ref{lemmafconstant}, and hence $wcdim(G)=3$. This is consistent with the result we would have obtained had we used Theorem \ref{thmGabriella}. 
\end{example}

Problems may arise if a \emph{SGGC}, $G$, were presented in some unrecognizable form. In that case, the problem of finding the well-covered dimension of $G$ is comparable to the problem of finding a minimum clique cover of $G$. 

Let $G$ be a \emph{SCCG} with $sc(G)=k,$ for some $k \in \mathbb{N}$, and let $C_i \in \mathcal{C}(G)$. Observe that the well-covered dimension of $G$ does not depend on $|C_i|$. Informally speaking, if we let $|C_i|\rightarrow \infty$, the well-covered dimension $G$ is still $k$. 

\begin{example}
Each graph in Figure \ref{exampleinf} is a \emph{SCCG} with simplicial clique number 2. So the well-covered dimension of each of these graphs is 2, although the set of simplicial cliques is distinct in each case.
\begin{center}
\begin{figure}[H]
{\begin{tikzpicture}[x=.4in, y=.4in]
\vertex (w1) at (0,0) [style={fill=blue}]{};
\vertex (w2) at (1,0)[style={fill=red}]{};
\vertex (w5) at (-0.5,1)[style={fill=blue}]{};
\vertex (w3) at (2,0)[style={fill=red}]{};
\vertex (w4) at (-1,0)[style={fill=blue}]{};
\path 
(w4) edge (w1)
(w1) edge (w2)
(w2) edge (w3)
(w5) edge (w1)
(w5) edge (w4);
\end{tikzpicture}}\hspace{.1in}  \hspace{.1in}{\begin{tikzpicture}[x=.4in, y=.4in]
\vertex (w1) at (0,0) [label=left:{}, style={fill=blue}]{};
\vertex (w2) at (1,0)[label=right:{}, style={fill=red}]{};
\vertex (w3) at (2,0)[style={fill=red}]{};
\vertex (w4) at (-1,0)[style={fill=blue}]{};
\vertex (w5) at (-1,1)[style={fill=blue}]{};
\vertex (w6) at (0,1)[style={fill=blue}]{};
\path 
(w5) edge (w6)
(w5) edge (w1)
(w6) edge (w1)
(w5) edge (w4)
(w6) edge (w4)
(w1) edge (w4)
(w2) edge (w3)
(w1) edge (w2);
\end{tikzpicture}}\hspace{.3in}
{\begin{tikzpicture}[x=.4in, y=.4in]
\vertex (w1) at (0,0) [style={fill=blue}]{};
\vertex (w2) at (1,0)[style={fill=red}]{};
\vertex (w3) at (2,0)[style={fill=red}]{};
\vertex (w4) at (-1,0)[style={fill=blue}]{};
\vertex (w5) at (-1.2,.8)[style={fill=blue}]{};
\vertex (w6) at (0.2,.8)[style={fill=blue}]{};
\vertex (w7) at (-0.5,1.5)[style={fill=blue}]{};
\path
(w6) edge (w7)
(w5) edge (w6)
(w5) edge (w7) 
(w1) edge (w5)
(w1) edge (w6)
(w1) edge (w7)
(w4) edge (w5)
(w4) edge (w6)
(w4) edge (w7)
(w4) edge (w1)
(w1) edge (w2)
(w2) edge (w3);
\end{tikzpicture}}\raisebox{.4in}{\Huge{$\hdots$}}
\caption{An infinite family of \emph{SCCG}s with $wcdim = 2$.}
\label{exampleinf}
\end{figure}
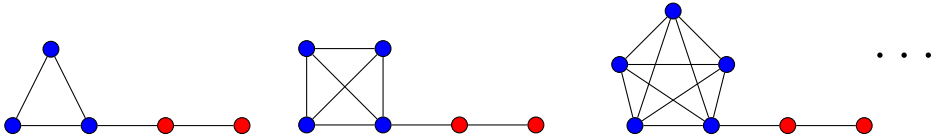
\end{center}
Thus, there is an infinite number of \emph{SCCG}s with well-covered dimension any positive integer. 
\end{example}

There are other ways of obtaining a family of \emph{SCCG}s with some desired well-covered dimension. In particular, if we add or remove edges between vertices of a given \emph{SCCG}, making sure that we do not create or delete a simplicial clique, we obtain a family of \emph{SCCG}s with the same well-covered dimension. However, this resulting family might not be infinite. 

\begin{example}
Consider the following \emph{SCCG}. 
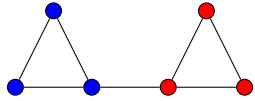
\[ 
\begin{tikzpicture}[x=.4in, y=.4in]
\vertex (w1) at (0,0) [style={fill=blue}]{};
\vertex (w2) at (1,0)[style={fill=red}]{};
\vertex (w5) at (-0.5,1)[style={fill=blue}]{};
\vertex (w3) at (2,0)[style={fill=red}]{};
\vertex (w4) at (-1,0)[style={fill=blue}]{};
\vertex (w6) at (1.5,1)[style={fill=red}]{};
\path 
(w4) edge (w1)
(w1) edge (w2)
(w2) edge (w3)
(w5) edge (w1)
(w5) edge (w4)
(w6) edge (w3)
(w6) edge (w2);
\end{tikzpicture}
\] 
For this particular graph, we may obtain only finitely many \emph{SCCG}s by removing and adding edges. Examples of these graphs are illustrated below.
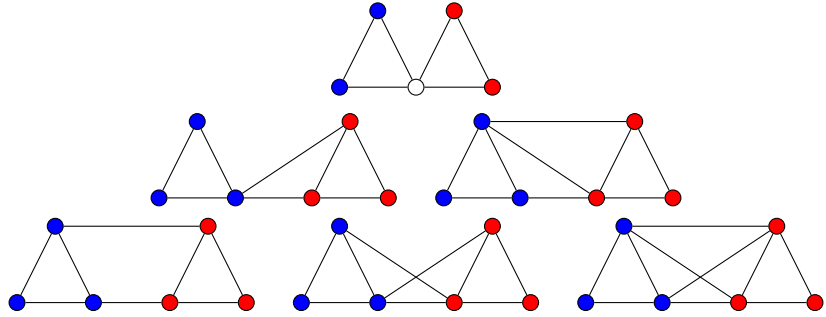
\[ 
{\begin{tikzpicture}[x=.4in, y=.4in]
\vertex (w1) at (0,0){};
\vertex (w2) at (1,0)[style={fill=red}]{};
\vertex (w3) at (0.5,1)[style={fill=red}]{};
\vertex (w4) at (-0.5,1)[style={fill=blue}]{};
\vertex (w5) at (-1,0)[style={fill=blue}]{};
\path 
(w1) edge (w2)
(w2) edge (w3)
(w1) edge (w3)
(w1) edge (w4)
(w4) edge (w5)
(w1) edge (w5);
\end{tikzpicture}}
\]
\[ 
{\begin{tikzpicture}[x=.4in, y=.4in]
\vertex (w1) at (0,0) [style={fill=blue}]{};
\vertex (w2) at (1,0)[style={fill=red}]{};
\vertex (w5) at (-0.5,1)[style={fill=blue}]{};
\vertex (w3) at (2,0)[style={fill=red}]{};
\vertex (w4) at (-1,0)[style={fill=blue}]{};
\vertex (w6) at (1.5,1)[style={fill=red}]{};
\path 
(w4) edge (w1)
(w1) edge (w2)
(w2) edge (w3)
(w5) edge (w1)
(w5) edge (w4)
(w6) edge (w3)
(w6) edge (w2)
(w1) edge (w6);
\end{tikzpicture}}\hspace{.2in}{\begin{tikzpicture}[x=.4in, y=.4in]
\vertex (w1) at (0,0) [style={fill=blue}]{};
\vertex (w2) at (1,0)[style={fill=red}]{};
\vertex (w5) at (-0.5,1)[style={fill=blue}]{};
\vertex (w3) at (2,0)[style={fill=red}]{};
\vertex (w4) at (-1,0)[style={fill=blue}]{};
\vertex (w6) at (1.5,1)[style={fill=red}]{};
\path 
(w4) edge (w1)
(w1) edge (w2)
(w2) edge (w3)
(w5) edge (w1)
(w5) edge (w4)
(w6) edge (w3)
(w6) edge (w2)
(w2) edge (w5)
(w5) edge (w6);
\end{tikzpicture}}
\]
\[ 
{\begin{tikzpicture}[x=.4in, y=.4in]
\vertex (w1) at (0,0) [style={fill=blue}]{};
\vertex (w2) at (1,0)[style={fill=red}]{};
\vertex (w5) at (-0.5,1)[style={fill=blue}]{};
\vertex (w3) at (2,0)[style={fill=red}]{};
\vertex (w4) at (-1,0)[style={fill=blue}]{};
\vertex (w6) at (1.5,1)[style={fill=red}]{};
\path 
(w4) edge (w1)
(w1) edge (w2)
(w2) edge (w3)
(w5) edge (w1)
(w5) edge (w4)
(w6) edge (w3)
(w6) edge (w2)
(w5) edge (w6);
\end{tikzpicture}}\hspace{.2in}{\begin{tikzpicture}[x=.4in, y=.4in]
\vertex (w1) at (0,0) [style={fill=blue}]{};
\vertex (w2) at (1,0)[style={fill=red}]{};
\vertex (w5) at (-0.5,1)[style={fill=blue}]{};
\vertex (w3) at (2,0)[style={fill=red}]{};
\vertex (w4) at (-1,0)[style={fill=blue}]{};
\vertex (w6) at (1.5,1)[style={fill=red}]{};
\path 
(w4) edge (w1)
(w1) edge (w2)
(w2) edge (w3)
(w5) edge (w1)
(w5) edge (w4)
(w6) edge (w3)
(w6) edge (w2)
(w2) edge (w5)
(w1) edge (w6);
\end{tikzpicture}}\hspace{.2in}{\begin{tikzpicture}[x=.4in, y=.4in]
\vertex (w1) at (0,0) [style={fill=blue}]{};
\vertex (w2) at (1,0)[style={fill=red}]{};
\vertex (w5) at (-0.5,1)[style={fill=blue}]{};
\vertex (w3) at (2,0)[style={fill=red}]{};
\vertex (w4) at (-1,0)[style={fill=blue}]{};
\vertex (w6) at (1.5,1)[style={fill=red}]{};
\path 
(w4) edge (w1)
(w1) edge (w2)
(w2) edge (w3)
(w5) edge (w1)
(w5) edge (w4)
(w6) edge (w3)
(w6) edge (w2)
(w2) edge (w5)
(w1) edge (w6)
(w5) edge (w6);
\end{tikzpicture}}
\] 

All of these graphs have well-covered dimension 2 and the same set of simplicial cliques. 
\end{example}

\section{The well-covered dimension of simplicial clique sums}

In this section, we obtain a class of graphs, with well-covered dimension equal to their simplicial clique number, that contains both chordal graphs and \emph{SCCG}s.

\begin{definition}\label{defsum}
Let $G_1$ and $G_2$ be sugraphs of a graph $\mathcal{G}$ such that $\mathcal{C}(G_1)$ and $\mathcal{C}(G_2)$ are non-empty. We say that $\mathcal{G}$ is the simplicial clique sum (\emph{SCS}) of $G_1$ and $G_2$ if
\begin{enumerate}
\item $V(G_1) \cup V(G_2)=V(\mathcal{G}),$
\item $E(G_1) \cup E(G_2)=E(\mathcal{G}),$ and
\item $V(G_1) \cap V(G_2)$ is a simplicial clique of $G_1$, $G_2$ and $\mathcal{G}$.
\end{enumerate}
\end{definition} 

\begin{remark}\label{intact}
Let $\mathcal{G}$ be the \emph{SCS} of $G_1$ and $G_2$. Then, for any $u \in V(G_1)- \left( V(G_1) \cap V(G_2)\right)$ and any $v \in V(G_2)-  \left(V(G_1) \cap V(G_2)\right)$, $uv \notin E(\mathcal{G})$. 
\end{remark}

Let $G$ be a chordal graph with $\mathcal{C}(G) \neq \emptyset$ and $G'$ be a \emph{SCCG}. Then, $K_n,$ for some $n \in \mathbb{N},$ is a subgraph of $G,$ and thus, $G$ may be understood to be the \emph{SCS} of $G_1=K_n$ and $G_2=G.$ Likewise, $G'$ can be understood to be the \emph{SCS} of one of its complete subgraphs with itself. This is a remarkable fact because it allows us to view the \emph{SCS} class of graphs, as a class that contains all chordal graphs and all \emph{SCCG}s. Figure \ref{figSCSplusSCCG} is an example of the \emph{SCS} of a \emph{SCCG} that is not chordal (red) and a chordal graph that is not a \emph{SCCG} (black). The yellow simplicial clique is their intersection. 

\begin{center}
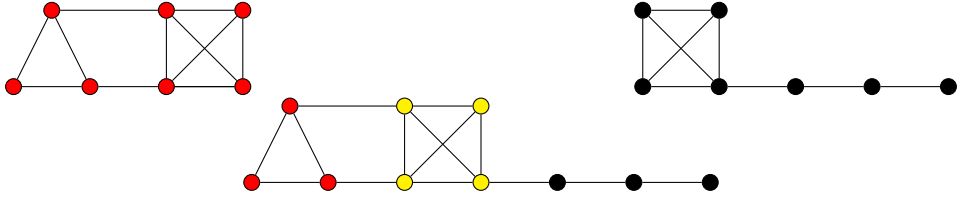
\begin{figure}[H]
{\begin{tikzpicture}[x=.4in, y=.4in]
\vertex (w1) at (0,0) [style={fill=red}]{};
\vertex (w2) at (1,0)[style={fill=red}]{};
\vertex (w5) at (-0.5,1)[style={fill=red}]{};
\vertex (w3) at (2,0)[style={fill=red}]{};
\vertex (w4) at (-1,0)[style={fill=red}]{};
\vertex (w6) at (1,1)[style={fill=red}]{};
\vertex (w7) at (2,1)[style={fill=red}]{};
\path
(w4) edge (w1)
(w1) edge (w2)
(w2) edge (w3)
(w5) edge (w1)
(w2) edge (w3)
(w3) edge (w6)
(w6) edge (w2)
(w5) edge (w6)
(w3) edge (w7)
(w6) edge (w7)
(w2) edge (w7)
(w5) edge (w4); \end{tikzpicture}}\hfill{\begin{tikzpicture}[x=.4in, y=.4in]

\vertex (w2) at (1,0)[style={fill=black}]{};
\vertex (w3) at (2,0)[style={fill=black}]{};
\vertex (w6) at (1,1)[style={fill=black}]{};
\vertex (w7) at (3,0)[style={fill=black}]{};
\vertex (w8) at (4,0)[style={fill=black}]{};
\vertex (w9) at (5,0)[style={fill=black}]{};
\vertex (w10) at (2,1)[style={fill=black}]{};
\path 
(w2) edge (w10)
(w3) edge (w10)
(w6) edge (w10)
(w2) edge (w3)
(w6) edge (w3)
(w6) edge (w2)
(w3) edge (w7)
(w7) edge (w8)
(w8) edge (w9);
\end{tikzpicture}}
 
\begin{tikzpicture}[x=.4in, y=.4in]
\vertex (w1) at (0,0) [style={fill=red}]{};
\vertex (w2) at (1,0)[style={fill=yellow}]{};
\vertex (w5) at (-0.5,1)[style={fill=red}]{};
\vertex (w3) at (2,0)[style={fill=yellow}]{};
\vertex (w4) at (-1,0)[style={fill=red}]{};
\vertex (w6) at (1,1)[style={fill=yellow}]{};
\vertex (w7) at (3,0)[style={fill=black}]{};
\vertex (w8) at (4,0)[style={fill=black}]{};
\vertex (w9) at (5,0)[style={fill=black}]{};
\vertex (w10) at (2,1)[style={fill=yellow}]{};

\path 
(w2) edge (w10)
(w6) edge (w10)
(w3) edge (w10)
(w4) edge (w1)
(w1) edge (w2)
(w2) edge (w3)
(w5) edge (w1)
(w5) edge (w4)
(w6) edge (w3)
(w6) edge (w2)
(w6) edge (w5)
(w3) edge (w7)
(w7) edge (w8)
(w8) edge (w9);
\end{tikzpicture}
\caption{\emph{SCS} of a \emph{SCCG} and a chordal graph.}
\label{figSCSplusSCCG}
\end{figure}
\end{center}

Next, is our main result on maximal independent sets of \emph{SCS}s.

\begin{theorem}\label{bigmis}
Let $\mathcal{G}$ be the \emph{SCS} of $G_1$ and $G_2$. Let $M_1$ be a \emph{MIS} of $G_1$, $M_2$ be a \emph{MIS} of $G_2$, and $M_1 \cap M_2 = \{v\}$, for some $v \in V(G_1) \cap V(G_2)$. Then, $\mathcal{M}$ is a \emph{MIS} of $\mathcal{G}$ if and only if $\mathcal{M}=M_1 \cup M_2.$
\end{theorem}

\begin{proof}
This proof consists of two parts: verifying that (I) $M_1 \cup M_2,$ as described in the hypothesis of the theorem, is a \emph{MIS} of $\mathcal{G}$ and that (II) all \emph{MIS}s of $\mathcal{G}$ are of this form. 

(I) Suppose, for a contradiction, that $M_1 \cup M_2$ is not a \emph{MIS} of $\mathcal{G}$. Then, either $M_1 \cup M_2$ is dependent or $M_1 \cup M_2$ is independent but not maximal.  In the first case, there exist $u,v \in M_1 \cup M_2$ such that $uv \in E(\mathcal{G}).$ There are two possibilities: (1) $u,v \in M_i,$ where $i=1,2$ or (2) $u \in M_1-M_2$ and $v \in M_2-M_1.$ But (1) implies that $uv \in E(G_i)$ so that $M_i$ is a dependent set of $G_i$, and (2) contradicts Remark \ref{intact}, since $M_1-M_2 \subset V(G_1) - V(G_1) \cap V(G_2)$ and $M_2 -M_1 \subset V(G_2) - V(G_1) \cap V(G_2).$ If $M_1 \cup M_2$ is independent but not maximal, there is $w \in V(\mathcal{G})$ such that $M_1 \cup M_2 \cup \{w\}$ is independent. Say, $M_1 \cup \{w\}$ is an independent set of $G_1.$ But then, $M_1$ is not maximal in $G_1.$ Therefore, $M_1 \cup M_2$ is a \emph{MIS} of $\mathcal{G}.$

(II) Any \emph{MIS} of $\mathcal{G}$ must contain vertices of $G_1$ and $G_2$. This is because $\mathcal{C}(G_1)$ and $\mathcal{C}(G_2)$ are non-empty, and any \emph{MIS} of any graph must contain a vertex per simplicial clique. Thus, any \emph{MIS} of $\mathcal{G}$ can be expressed as $I_1 \cup I_2$, where $I_1 \subset V(G_1)$ and $I_2 \subset V(G_2)$ are non-empty. Without loss of generality, suppose that $I_1$ is not a \emph{MIS} of $G_1.$ Clearly, if $I_1$ is dependent, so is $I_1 \cup I_2,$ so we discard that possibility. Suppose that $I_1$ is independent but not maximal in $G_1.$ Then, $I_1 \cup \{y\}$ is independent, for some $y \in V(G_1).$ But then, if $I_2 \cup \{y\}$ is independent, so is $I_1 \cup I_1 \cup \{y\}.$ Otherwise, $I_2 \cup \{y\}$ is dependent and so is $I_1 \cup I_2 \cup \{y\}.$ Thus, $I_1$ and $I_2$ must be \emph{MIS}s of $G_1$ and $G_2.$ By Remark \ref{intact}, the only vertices at which $I_1$ and $I_2$ can interesect belong to $V(G_1) \cap V(G_2).$ Observe that $I_1 \cap I_2$ can contain at most one vertex from $V(G_1) \cap V(G_2)$, since any two vertices in this set are adjacent. Observe too that if $I_1 \cap I_2 = \emptyset,$ then each vertex in $V(G_1) \cap V(G_2)$ must be adjacent to vertices outside of $V(G_1) \cap V(G_2).$ But this forces $V(G_1) \cap V(G_2)$ to be non-simplicial. Therefore, every \emph{MIS} of $\mathcal{G}$ is of the form $I_1 \cup I_2,$ where $I_1$ and $I_2$ are \emph{MIS}s of $G_1$ and $G_2$ that intersect at a single vertex from $V(G_1) \cap V(G_2).$
\end{proof}

\begin{corollary}
Let $\mathcal{G}$ be the \emph{SCS} of $G_1$ and $G_2$, and let $c=|V(G_1) \cap V(G_2)|$. Let $v_i \in V(G_1) \cap V(G_2),$ let $l_i$ be the number of \emph{MIS}s of $G_1$ that contain $v_i,$ and let $m_i$ be the number of \emph{MIS}s of $G_2$ that contain $v_i.$ Then, $\mathcal{G}$ has exactly $$\sum_{i=1}^c l_im_i$$ maximal independent sets. 
\end{corollary}

\begin{proof}
We know the structure of the \emph{MIS}s of $\mathcal{G}$ from Theorem \ref{bigmis}. Observe that per each of the $l_i$ \emph{MIS}s of $G_1,$ we can form $m_i$ \emph{MIS}s of $\mathcal{G}$ that contain $v_i.$ That is, we can form $l_im_i$ \emph{MIS}s of $\mathcal{G}$ that contain $v_i.$ Since there are $c$ such $v_i,$ the result follows. 
\end{proof}

\begin{theorem}\label{wcdimSCS}
Let $\mathcal{G}$ be the \emph{SCS} of $G_1$ and $G_2.$ Then, 
\[
wcdim(\mathcal{G})=wcdim(G_1)+wcdim(G_2)-1.
\]
\end{theorem}

\begin{proof}
Let $\mathcal{V}$ be the well-covered space of $\mathcal{G},$ $f \in \mathcal{V}$ and $u \in V(G_1) \cap V(G_2).$ Define $W_{G_1}$ to be the vector space of all $f$ that assign zero to vertices outside of $V(G_1)$ with the added property that for any $g \in W_{G_1},$ $g(u)=\frac{f(u)}{2}.$ Similarly, define $W_{G_2}$ to be the vector space of all $f$ that assign zero to vertices outside of $V(G_2)$ and such that for any $h \in W_{G_2},$ $h(u)=\frac{f(u)}{2}.$ By Theorem \ref{bigmis}, we know what the \emph{MIS}s of $\mathcal{G}$ look like. Namely $\mathcal{M}=M_1 \cup M_2,$ where $M_1 \cap M_2 =\{s\}$ for some $s \in V(G_1) \cap V(G_2).$  Observe that all functions in $W_{G_1}$ and $W_{G_2}$ have domain $V(\mathcal{G}),$ which is the domain of $f,$ and thus addition of these functions is defined. Then, for any $g \in W_{G_1}$ and any $h \in W_{G_2},$

\begin{align*} 
\sum_{v \in \mathcal{M}} (g+h)(v) &= \sum_{v \in \mathcal{M}} g(v) + \sum_{v \in \mathcal{M}} h(v) \\ &= \bigg( \sum_{v \in \mathcal{M}-\{s\}} g(v) + g(s) \bigg) + \bigg(\sum_{v \in \mathcal{M}-\{s\}} h(v) + h(s) \bigg) \\ &= \sum_{v \in \mathcal{M}-\{s\}} g(v) + \sum_{v \in \mathcal{M}-\{s\}} h(v) + f(s) \\ &= \sum_{v \in \mathcal{M}} f(v).
\end{align*}

This shows that $W_{G_1} + W_{G_2} = \mathcal{V}$. Since $W_{G_1} \cap W_{G_2}$ is the set of all $f$ that are zero everywhere, except at $V(G_1) \cap V(G_2)$, where they are constant, we get that $dim(\mathcal{V})=dim(W_{G_1})+dim(W_{G_2})-1$, and the result follows.
\end{proof}

Now we are finally able to state and prove our result generalizing Brown and Nowakowski's theorem on the well-covered dimension of chordal graphs.

\begin{theorem}\label{mainthm}
Let $\mathcal{G}$ be the \emph{SCS} of $G_1,$ a \emph{SCCG}, and $G_2,$ a chordal graph. Then, $wcdim(\mathcal{G})=sc(\mathcal{G})$.
\end{theorem}

\begin{proof}
By Theorem \ref{wcdimSCS}, $wcdim(\mathcal{G})=sc(G_1)+sc(G_2)-1=sc(\mathcal{G})$.
\end{proof}

The question is now whether there are any graphs that have minimal well-covered dimension and are not \emph{SCS} graphs. In order to answer this question, we turn to the study of the well-covered dimension of Sierpinski gasket graphs, of which there are infinitely many. The Sierpinski gasket graphs are not part of the family of graphs to which Theorem \ref{mainthm} applies. In spite of this, all Sierpinski gasket graphs have well-covered dimension the simplicial clique number. 

\section{The well-covered dimension of Sierpinski gasket graphs}

In this section, we study the well-covered dimension of the Sierpinski gasket graph, which we denote by $S_n$, for any $n \in \mathbb{N}$. The Sierpinski gasket graph is constructed recursively, in the same way the Sierpinski gasket is constructed. \\

\begin{center}
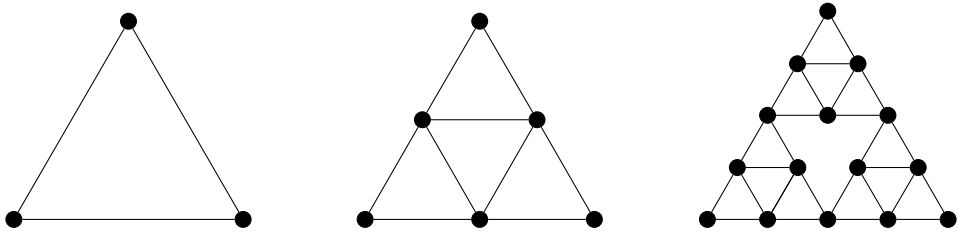
\begin{figure}[H]
{\begin{tikzpicture}[x=.6in, y=.6in]
\vertex (w1) at (0,0) [label=below:$$, style={fill=black}]{};
\vertex (w2) at (2,0)[label=below:$$, style={fill=black}]{};
\vertex (w3) at (1,1.73)[label=below:$$, style={fill=black}]{};
\path
(w1) edge (w2)
(w2) edge (w3)
(w1) edge (w3);
\end{tikzpicture}}\hfill{
\begin{tikzpicture}[x=.6in, y=.6in]
\vertex (w1) at (0,0) [label=below:$$, style={fill=black}]{};
\vertex (w2) at (1,0)[label=below:$$, style={fill=black}]{};
\vertex (w3) at (2,0)[label=below:$$, style={fill=black}]{};
\vertex (w4) at (0.5,0.87)[label=below:$$, style={fill=black}]{};
\vertex (w5) at (1.5,0.87)[label=below:$$, style={fill=black}]{};
\vertex (w6) at (1,1.73)[label=below:$$, style={fill=black}]{};
\path 
(w1) edge (w2)
(w2) edge (w3)
(w1) edge (w4)
(w4) edge (w5)
(w2) edge (w5)
(w2) edge (w4)
(w3) edge (w5)
(w4) edge (w6)
(w5) edge (w6);
\end{tikzpicture}}\hfill{\begin{tikzpicture}[x=.18in, y=.18in]
\vertex (w1) at (0,0) [label=below:$$, style={fill=black}]{};
\vertex (w2) at (7,0)[label=below:$$, style={fill=black}]{};
\vertex (w3) at (3.5,6.06)[label=below:$$, style={fill=black}]{};
\vertex (w4) at (3.5,0)[label=below:$$, style={fill=black}]{};
\vertex (w5) at (1.75,3.03)[label=below:$$, style={fill=black}]{};
\vertex (w6) at (5.25,3.03)[label=below:$$, style={fill=black}]{};
\vertex (w7) at (1.75,0)[label=below:$$, style={fill=black}]{};
\vertex (w8) at (5.25,0)[label=below:$$, style={fill=black}]{};
\vertex (w9) at (0.87,1.51)[label=below:$$, style={fill=black}]{};
\vertex (w10) at (6.13,1.51)[label=below:$$, style={fill=black}]{};
\vertex (w11) at (2.63,1.51)[label=below:$$, style={fill=black}]{};
\vertex (w12) at (4.37,1.51)[label=below:$$, style={fill=black}]{};
\vertex (w13) at (3.5,3.03)[label=below:$$, style={fill=black}]{};
\vertex (w14) at (2.62,4.53)[label=below:$$, style={fill=black}]{};
\vertex (w15) at (4.38,4.53)[label=below:$$, style={fill=black}]{};
\path 
(w1) edge (w7)
(w4) edge (w7)
(w4) edge (w8)
(w8) edge (w2)
(w1) edge (w9)
(w9) edge (w7)
(w11) edge (w7)
(w11) edge (w4)
(w11) edge (w7)
(w4) edge (w12)
(w8) edge (w12)
(w8) edge (w10)
(w2) edge (w10)
(w9) edge (w11)
(w12) edge (w10)
(w9) edge (w5)
(w5) edge (w11)
(w12) edge (w6)
(w6) edge (w10)
(w5) edge (w13)
(w13) edge (w6)
(w5) edge (w14)
(w14) edge (w13)
(w14) edge (w15)
(w15) edge (w13)
(w6) edge (w15)
(w14) edge (w3)
(w3) edge (w15);
\end{tikzpicture}}
\caption{The first three Sierpinski graphs, $S_1,S_2$ and $S_3$.}
\label{fig1}
\end{figure}
\end{center}

From Figure \ref{fig1}, it should be clear that $S_1$ is a $K_3$ and that $S_2$ is both a \emph{SCCG} and a chordal graph with $sc(S_2)=3.$ Thus, $wcdim(S_1)=1$ and $wcdim(S_2)=3.$ However, for $n \in \mathbb{N}_{\geq 3}$, $S_n$ is not a \emph{SCCG}, not a chordal graph, and not the \emph{SCS} of a chordal graph and a \emph{SCCG}. All we know about the well-covered dimension of $S_n$ for $n \in \mathbb{N}_{\geq 3}$ is stated below and is a direct consequence of Lemma 10 in \cite{BN}.

\begin{remark}\label{LBOUND}
For $n \in \mathbb{N}_{\geq3},$  $wcdim(S_n) \geq 3.$  
\end{remark}

From the recursive construction of $S_n$, it follows that $S_n$ has all of its predecessors as subgraphs. In particular, $S_n$ has exactly three $S_{n-1}$ subgraphs. Observe that for $n \in \mathbb{N}_{\geq 3},$ $S_n$ has \emph{sides} that are paths of length at least $5,$ and that an $S_{n-1}$ subgraph of $S_n$ has \emph{corners} that are the \emph{former simplicial cliques} of the $S_{n-1}$ iterate. An $S_{n-1}$ subgraph of $S_n$ has three corners, only one of which is a simplicial clique of $S_n.$ These ideas are exemplified in
Figure \ref{B}. In the graph to the left, a \emph{side} of $S_4$ is colored yellow and an $S_3$ subgraph of $S_4$ is colored blue. In the graph to the right, the simplicial clique corner of an $S_3$ subgraph of $S_4$ is colored red.

\begin{center}
\begin{figure}[H]
{\begin{tikzpicture}[x=.26in, y=.26in]
\vertex (w1) at (0,0)[label=below:$$, style={fill=yellow}]{};
\vertex (w2) at (7,0)[label=below:$$, style={fill=blue}]{};
\vertex (w3) at (3.5,6.06)[label=below:$$, style={fill=yellow}]{};
\vertex (w4) at (3.5,0)[label=below:$$, style={fill=blue}]{};
\vertex (w5) at (1.75,3.03)[label=below:$$, style={fill=yellow}]{};
\vertex (w6) at (5.25,3.03)[label=below:$$, style={fill=blue}]{};
\vertex (w7) at (1.75,0)[label=below:$$, style={fill=black}]{};
\vertex (w8) at (5.25,0)[label=below:$$, style={fill=blue}]{};
\vertex (w9) at (0.87,1.51)[label=below:$$, style={fill=yellow}]{};
\vertex (w10) at (6.13,1.51)[label=below:$$, style={fill=blue}]{};
\vertex (w11) at (2.63,1.51)[label=right:$$, style={fill=black}]{};
\vertex (w12) at (4.37,1.51)[label=below:$$, style={fill=blue}]{};
\vertex (w13) at (2.62,4.53)[label=below:$$, style={fill=yellow}]{};
\vertex (w14) at (4.38,4.53)[label=below:$$, style={fill=black}]{};
\vertex (w15) at (3.49,3.03)[label=below:$$, style={fill=black}]{};
\vertex (w16) at (2.62,3.03)[label=below:$$, style={fill=black}]{};
\vertex (w17) at (4.38,3.03)[label=below:$$, style={fill=black}]{};
\vertex (w18) at (.87, 0)[label=below:$$, style={fill=black}]{};
\vertex (w19) at (2.62,0)[label=below:$$, style={fill=black}]{};
\vertex (w20) at (4.38,0)[label=below:$$, style={fill=blue}]{};
\vertex (w21) at (6.12,0)[label=below:$$, style={fill=blue}]{};
\vertex (w22) at (0.435,0.753)[label=below:$$, style={fill=yellow}]{};
\vertex (w23) at (1.305,0.753)[label=below:$$, style={fill=black}]{};
\vertex (w24) at (2.175,0.753)[label=below:$$, style={fill=black}]{};
\vertex (w25) at (3.045,0.753)[label=below:$$, style={fill=black}]{};
\vertex (w26) at (3.915,0.753)[label=below:$$, style={fill=blue}]{};
\vertex (w27) at (4.785,0.753)[label=below:$$, style={fill=blue}]{};
\vertex (w28) at (5.655,0.753)[label=below:$$, style={fill=blue}]{};
\vertex (w29) at (6.525,0.753)[label=below:$$, style={fill=blue}]{};
\vertex (w30) at (1.75,1.51)[label=below:$$, style={fill=black}]{};
\vertex (w31) at (5.25,1.51)[label=below:$$, style={fill=blue}]{};
\vertex (w32) at (1.305,2.26)[label=below:$$, style={fill=yellow}]{};
\vertex (w33) at (2.175,2.26)[label=right:$$, style={fill=black}]{};
\vertex (w34) at (4.805,2.26)[label=below:$$, style={fill=blue}]{};
\vertex (w35) at (5.675,2.26)[label=below:$$, style={fill=blue}]{};
\vertex (w36) at (2.185,3.78)[label=below:$$, style={fill=yellow}]{};
\vertex (w37) at (3.055,3.78)[label=below:$$, style={fill=black}]{};
\vertex (w38) at (3.925,3.78)[label=below:$$, style={fill=black}]{};
\vertex (w39) at (4.795,3.78)[label=below:$$, style={fill=black}]{};
\vertex (w40) at (3.49,4.53)[label=below:$$, style={fill=black}]{};
\vertex (w41) at (3.055,5.28)[label=below:$$, style={fill=yellow}]{};
\vertex (w42) at (3.925,5.28)[label=below:$$, style={fill=black}]{};
\path
(w1) edge (w18)
(w18) edge (w7)
(w7) edge (w19)
(w19) edge (w4)
(w4) edge (w20)
(w20) edge (w8)
(w8) edge (w21)
(w21) edge (w2)
(w1) edge (w22)
(w22) edge (w23)
(w23) edge (w7)
(w7) edge (w24)
(w24) edge (w19)
(w19) edge (w25)
(w25) edge (w4)
(w4) edge (w26)
(w26) edge (w20)
(w20) edge (w27)
(w27) edge (w8)
(w28) edge (w8)
(w28) edge (w21)
(w21) edge (w29)
(w29) edge (w2)
(w18) edge (w22)
(w18) edge (w23)
(w22) edge (w9)
(w23) edge (w9)
(w9) edge (w32)
(w32) edge (w5)
(w5) edge (w36)
(w36) edge (w13)
(w13) edge (w36)
(w13) edge (w41)
(w41) edge (w3)
(w3) edge (w42)
(w14) edge (w42)
(w14) edge (w39)
(w39) edge (w6)
(w6) edge (w35)
(w35) edge (w10)
(w10) edge (w29)
(w32) edge (w30)
(w9) edge (w30)
(w11) edge (w30)
(w11) edge (w24)
(w30) edge (w33)
(w33) edge (w11)
(w32) edge (w33)
(w5) edge (w33)
(w11) edge (w25)
(w24) edge (w25)
(w26) edge (w12)
(w12) edge (w34)
(w34) edge (w6)
(w34) edge (w35)
(w12) edge (w27)
(w26) edge (w27)
(w12) edge (w31)
(w31) edge (w10)
(w10) edge (w28)
(w28) edge (w29)
(w31) edge (w34)
(w31) edge (w35)
(w6) edge (w17)
(w17) edge (w15)
(w15) edge (w16)
(w16) edge (w5)
(w36) edge (w16)
(w16) edge (w37)
(w37) edge (w15)
(w36) edge (w37)
(w13) edge (w37)
(w13) edge (w40)
(w40) edge (w14)
(w14) edge (w38)
(w38) edge (w15)
(w38) edge (w17)
(w17) edge (w39)
(w38) edge (w39)
(w41) edge (w40)
(w41) edge (w42)
(w40) edge (w42);
\end{tikzpicture}}\hspace{.2in}{\begin{tikzpicture}[x=.26in, y=.26in]
\vertex (w1) at (0,0)[label=below:$$, style={fill=yellow}]{};
\vertex (w2) at (7,0)[label=below:$$, style={fill=red}]{};
\vertex (w3) at (3.5,6.06)[label=below:$$, style={fill=yellow}]{};
\vertex (w4) at (3.5,0)[label=below:$$, style={fill=blue}]{};
\vertex (w5) at (1.75,3.03)[label=below:$$, style={fill=yellow}]{};
\vertex (w6) at (5.25,3.03)[label=below:$$, style={fill=blue}]{};
\vertex (w7) at (1.75,0)[label=below:$$, style={fill=black}]{};
\vertex (w8) at (5.25,0)[label=below:$$, style={fill=blue}]{};
\vertex (w9) at (0.87,1.51)[label=below:$$, style={fill=yellow}]{};
\vertex (w10) at (6.13,1.51)[label=below:$$, style={fill=blue}]{};
\vertex (w11) at (2.63,1.51)[label=right:$$, style={fill=black}]{};
\vertex (w12) at (4.37,1.51)[label=below:$$, style={fill=blue}]{};
\vertex (w13) at (2.62,4.53)[label=below:$$, style={fill=yellow}]{};
\vertex (w14) at (4.38,4.53)[label=below:$$, style={fill=black}]{};
\vertex (w15) at (3.49,3.03)[label=below:$$, style={fill=black}]{};
\vertex (w16) at (2.62,3.03)[label=below:$$, style={fill=black}]{};
\vertex (w17) at (4.38,3.03)[label=below:$$, style={fill=black}]{};
\vertex (w18) at (.87, 0)[label=below:$$, style={fill=black}]{};
\vertex (w19) at (2.62,0)[label=below:$$, style={fill=black}]{};
\vertex (w20) at (4.38,0)[label=below:$$, style={fill=blue}]{};
\vertex (w21) at (6.12,0)[label=below:$$, style={fill=red}]{};
\vertex (w22) at (0.435,0.753)[label=below:$$, style={fill=yellow}]{};
\vertex (w23) at (1.305,0.753)[label=below:$$, style={fill=black}]{};
\vertex (w24) at (2.175,0.753)[label=below:$$, style={fill=black}]{};
\vertex (w25) at (3.045,0.753)[label=below:$$, style={fill=black}]{};
\vertex (w26) at (3.915,0.753)[label=below:$$, style={fill=blue}]{};
\vertex (w27) at (4.785,0.753)[label=below:$$, style={fill=blue}]{};
\vertex (w28) at (5.655,0.753)[label=below:$$, style={fill=blue}]{};
\vertex (w29) at (6.525,0.753)[label=below:$$, style={fill=red}]{};
\vertex (w30) at (1.75,1.51)[label=below:$$, style={fill=black}]{};
\vertex (w31) at (5.25,1.51)[label=below:$$, style={fill=blue}]{};
\vertex (w32) at (1.305,2.26)[label=below:$$, style={fill=yellow}]{};
\vertex (w33) at (2.175,2.26)[label=right:$$, style={fill=black}]{};
\vertex (w34) at (4.805,2.26)[label=below:$$, style={fill=blue}]{};
\vertex (w35) at (5.675,2.26)[label=below:$$, style={fill=blue}]{};
\vertex (w36) at (2.185,3.78)[label=below:$$, style={fill=yellow}]{};
\vertex (w37) at (3.055,3.78)[label=below:$$, style={fill=black}]{};
\vertex (w38) at (3.925,3.78)[label=below:$$, style={fill=black}]{};
\vertex (w39) at (4.795,3.78)[label=below:$$, style={fill=black}]{};
\vertex (w40) at (3.49,4.53)[label=below:$$, style={fill=black}]{};
\vertex (w41) at (3.055,5.28)[label=below:$$, style={fill=yellow}]{};
\vertex (w42) at (3.925,5.28)[label=below:$$, style={fill=black}]{};
\path
(w1) edge (w18)
(w18) edge (w7)
(w7) edge (w19)
(w19) edge (w4)
(w4) edge (w20)
(w20) edge (w8)
(w8) edge (w21)
(w21) edge (w2)
(w1) edge (w22)
(w22) edge (w23)
(w23) edge (w7)
(w7) edge (w24)
(w24) edge (w19)
(w19) edge (w25)
(w25) edge (w4)
(w4) edge (w26)
(w26) edge (w20)
(w20) edge (w27)
(w27) edge (w8)
(w28) edge (w8)
(w28) edge (w21)
(w21) edge (w29)
(w29) edge (w2)
(w18) edge (w22)
(w18) edge (w23)
(w22) edge (w9)
(w23) edge (w9)
(w9) edge (w32)
(w32) edge (w5)
(w5) edge (w36)
(w36) edge (w13)
(w13) edge (w36)
(w13) edge (w41)
(w41) edge (w3)
(w3) edge (w42)
(w14) edge (w42)
(w14) edge (w39)
(w39) edge (w6)
(w6) edge (w35)
(w35) edge (w10)
(w10) edge (w29)
(w32) edge (w30)
(w9) edge (w30)
(w11) edge (w30)
(w11) edge (w24)
(w30) edge (w33)
(w33) edge (w11)
(w32) edge (w33)
(w5) edge (w33)
(w11) edge (w25)
(w24) edge (w25)
(w26) edge (w12)
(w12) edge (w34)
(w34) edge (w6)
(w34) edge (w35)
(w12) edge (w27)
(w26) edge (w27)
(w12) edge (w31)
(w31) edge (w10)
(w10) edge (w28)
(w28) edge (w29)
(w31) edge (w34)
(w31) edge (w35)
(w6) edge (w17)
(w17) edge (w15)
(w15) edge (w16)
(w16) edge (w5)
(w36) edge (w16)
(w16) edge (w37)
(w37) edge (w15)
(w36) edge (w37)
(w13) edge (w37)
(w13) edge (w40)
(w40) edge (w14)
(w14) edge (w38)
(w38) edge (w15)
(w38) edge (w17)
(w17) edge (w39)
(w38) edge (w39)
(w41) edge (w40)
(w41) edge (w42)
(w40) edge (w42);
\end{tikzpicture}}
\caption{Subgraphs, sides and corners of $S_4.$}
\label{B}
\end{figure}
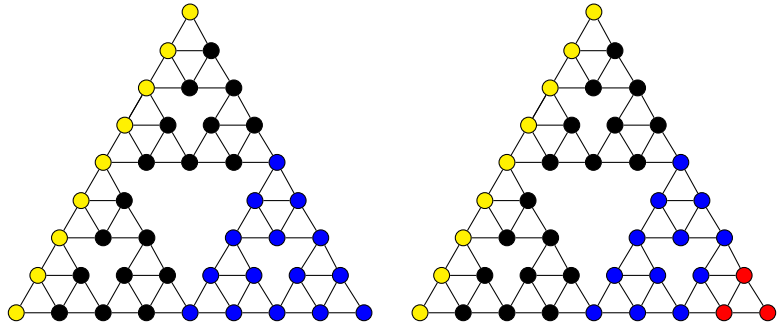
\end{center}

For any $n \in \mathbb{N}$, $|V(S_n)|=\frac{3(3^{n-1}+1)}{2}$ (see \cite{S}). For any $n \in \mathbb{N}_{\geq 3}$, $V(S_n)-\bigcup_{i=1}^3 C_i$ is the set of all vertices that do not belong to a simplicial clique of $S_n$. Note that $V(S_n)-\bigcup_{i=1}^3 C_i \neq \emptyset$ because when $n \in \mathbb{N}_{\geq 3}$, $|\bigcup_{i=1}^3 C_i|=9$ and $|V(S_n)| \geq 15$. The figure that follows shows $V(S_3)-\bigcup_{i=1}^3 C_i$ in blue.

\begin{center}
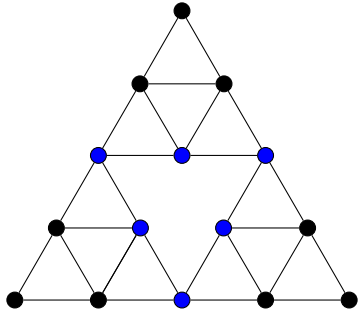
\begin{figure}[H]
\begin{tikzpicture}[x=.25in, y=.25in]
\vertex (w1) at (0,0)[label=below:$$, style={fill=black}]{};
\vertex (w2) at (7,0)[label=below:$$, style={fill=black}]{};
\vertex (w3) at (3.5,6.06)[label=below:$$, style={fill=black}]{};
\vertex (w4) at (3.5,0)[label=below:$$, style={fill=blue}]{};
\vertex (w5) at (1.75,3.03)[label=below:$$, style={fill=blue}]{};
\vertex (w6) at (5.25,3.03)[label=below:$$, style={fill=blue}]{};
\vertex (w7) at (1.75,0)[label=below:$$, style={fill=black}]{};
\vertex (w8) at (5.25,0)[label=below:$$, style={fill=black}]{};
\vertex (w9) at (0.87,1.51)[label=below:$$, style={fill=black}]{};
\vertex (w10) at (6.13,1.51)[label=below:$$, style={fill=black}]{};
\vertex (w11) at (2.63,1.51)[label=below:$$, style={fill=blue}]{};
\vertex (w12) at (4.37,1.51)[label=below:$$, style={fill=blue}]{};
\vertex (w13) at (3.5,3.03)[label=below:$$, style={fill=blue}]{};
\vertex (w14) at (2.62,4.53)[label=below:$$, style={fill=black}]{};
\vertex (w15) at (4.38,4.53)[label=below:$$, style={fill=black}]{};
\path 
(w1) edge (w7)
(w4) edge (w7)
(w4) edge (w8)
(w8) edge (w2)
(w1) edge (w9)
(w9) edge (w7)
(w11) edge (w7)
(w11) edge (w4)
(w11) edge (w7)
(w4) edge (w12)
(w8) edge (w12)
(w8) edge (w10)
(w2) edge (w10)
(w9) edge (w11)
(w12) edge (w10)
(w9) edge (w5)
(w5) edge (w11)
(w12) edge (w6)
(w6) edge (w10)
(w5) edge (w13)
(w13) edge (w6)
(w5) edge (w14)
(w14) edge (w13)
(w14) edge (w15)
(w15) edge (w13)
(w6) edge (w15)
(w14) edge (w3)
(w3) edge (w15);
\end{tikzpicture}
\caption{The set of all vertices not in a simplicial clique of $S_3$.}
\end{figure}
\end{center}

Moreover, if $w \in V(S_n)$ is non-simplicial, then $|N(w)|=4$. Otherwise, $w \in V(S_n)$ is simplicial and $|N(w)|=2$. In the figure below, we show the simplicial vertices of $S_1, S_2$ and $S_3$ in red, the neighborhood of a simplicial vertex in blue, and that of a non-simplicial vertex in green. 

\begin{center}
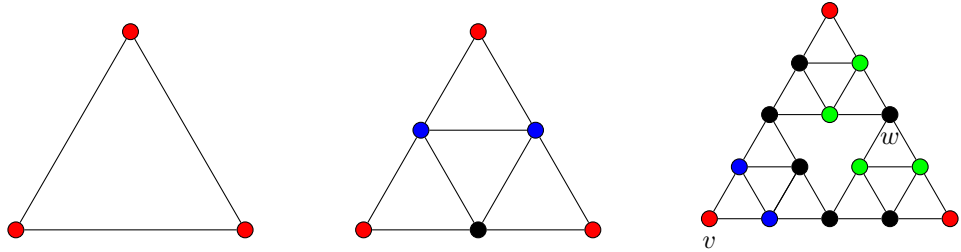
\begin{figure}[H]
{\begin{tikzpicture}[x=.6in, y=.6in]
\vertex (w1) at (0,0) [label=below:$$, style={fill=red}]{};
\vertex (w2) at (2,0)[label=below:$$, style={fill=red}]{};
\vertex (w3) at (1,1.73)[label=below:$$, style={fill=red}]{};
\path
(w1) edge (w2)
(w2) edge (w3)
(w1) edge (w3);
\end{tikzpicture}}\hfill{
\begin{tikzpicture}[x=.6in, y=.6in]
\vertex (w1) at (0,0) [label=below:$$, style={fill=red}]{};
\vertex (w2) at (1,0)[label=below:$$, style={fill=black}]{};
\vertex (w3) at (2,0)[label=below:$$, style={fill=red}]{};
\vertex (w4) at (0.5,0.87)[label=below:$$, style={fill=blue}]{};
\vertex (w5) at (1.5,0.87)[label=below:$$, style={fill=blue}]{};
\vertex (w6) at (1,1.73)[label=below:$$, style={fill=red}]{};
\path 
(w1) edge (w2)
(w2) edge (w3)
(w1) edge (w4)
(w4) edge (w5)
(w2) edge (w5)
(w2) edge (w4)
(w3) edge (w5)
(w4) edge (w6)
(w5) edge (w6);
\end{tikzpicture}}\hfill{\begin{tikzpicture}[x=.18in, y=.18in]
\vertex (w1) at (0,0)[label=below:$v$, style={fill=red}]{};
\vertex (w2) at (7,0)[label=below:$$, style={fill=red}]{};
\vertex (w3) at (3.5,6.06)[label=below:$$, style={fill=red}]{};
\vertex (w4) at (3.5,0)[label=below:$$, style={fill=black}]{};
\vertex (w5) at (1.75,3.03)[label=below:$$, style={fill=black}]{};
\vertex (w6) at (5.25,3.03)[label=below:$w$, style={fill=black}]{};
\vertex (w7) at (1.75,0)[label=below:$$, style={fill=blue}]{};
\vertex (w8) at (5.25,0)[label=below:$$, style={fill=black}]{};
\vertex (w9) at (0.87,1.51)[label=below:$$, style={fill=blue}]{};
\vertex (w10) at (6.13,1.51)[label=below:$$, style={fill=green}]{};
\vertex (w11) at (2.63,1.51)[label=below:$$, style={fill=black}]{};
\vertex (w12) at (4.37,1.51)[label=below:$$, style={fill=green}]{};
\vertex (w13) at (3.5,3.03)[label=below:$$, style={fill=green}]{};
\vertex (w14) at (2.62,4.53)[label=below:$$, style={fill=black}]{};
\vertex (w15) at (4.38,4.53)[label=below:$$, style={fill=green}]{};
\path 
(w1) edge (w7)
(w4) edge (w7)
(w4) edge (w8)
(w8) edge (w2)
(w1) edge (w9)
(w9) edge (w7)
(w11) edge (w7)
(w11) edge (w4)
(w11) edge (w7)
(w4) edge (w12)
(w8) edge (w12)
(w8) edge (w10)
(w2) edge (w10)
(w9) edge (w11)
(w12) edge (w10)
(w9) edge (w5)
(w5) edge (w11)
(w12) edge (w6)
(w6) edge (w10)
(w5) edge (w13)
(w13) edge (w6)
(w5) edge (w14)
(w14) edge (w13)
(w14) edge (w15)
(w15) edge (w13)
(w6) edge (w15)
(w14) edge (w3)
(w3) edge (w15);
\end{tikzpicture}}
\caption{Simplicial vertices and neighborhoods.}
\label{fig2}
\end{figure}
\end{center}

Let $n \in \mathbb{N}_{\geq 3}$ and $u,v$ be a pair of non-simplicial vertices that are adjacent. It is always the case that either $(1)$ $u$ and $v$ are adjacent to a unique third vertex $a$ or that $(2)$ $u$ and $v$ are adjacent to exactly two vertices $b$ and $c.$ This means that in case $(1),$ $|N[u] \cap N[v]|=|\{u,v,a\}|,$ while in case of $(2),$ $|N[u] \cap N[v]|=|\{u,v,b,c\}|.$ As a result, $6 \leq |N[\{u,v\}]| \leq 7,$ which shows that $V(S_n)-N[\{u,v\}] \neq \emptyset,$ since $|V(S_n)-N[\{u,v\}]| \geq 8.$ Additionally, $N[\{u,v\}]-\{u,v\} \neq \emptyset,$ since $|N[\{u,v\}]-\{u,v\}| \geq 4.$ 

\begin{lemma}[Lemma 9 in \cite{BN}]\label{RMKKBIR}
Let $u$ and $v$ be vertices of a graph $G.$ Suppose that $I$ is an independent set, and that $I \cup \{u\}$ and $I \cup \{v\}$ are \emph{MIS}s of $G.$ Then, for any well-covered weighting $f$ of $G,$ $f(u)=f(v)$.
\end{lemma}

\begin{theorem}[Theorem 5 in \cite{BK}]\label{BIRNAL}
Let $n \in \mathbb{N}_{\geq 5},$ $P_n$ be a path on $n$ vertices, and $f$ be a well-covered weighting of $P_n.$ Then, $f$ is constant on each of the two simplicial cliques of $P_n,$ while $f$ is zero at the remaining vertices of $P_n.$
\end{theorem}

In the proofs of results that follow, we make repeated use of the algorithm that is outlined below.

\begin{definition}[Page 2 of \cite{PP}]
A greedy algorithm is a tool for constructing maximal independent sets of graphs. It is executed in the following manner. Let $G$ be a graph.

\begin{enumerate}

\item Select $v_1 \in V(G)$ and set $I=\{v_1\}.$
\item Delete $v_1$ and its neighborhood in $G.$ The remaining vertices of $G$ induce a subgraph $G_{v_1}.$
\item  Select any vertex $v_2$ of $G_{v_1}$ and put it in the set $I.$
\item Repeat step (2) to obtain a subgraph $(G_{v_1})_{v_2}$ of $G_{v_1}.$
\item Continue this process until all vertices of $G$ either have been added to $I$ or have been deleted.

The resulting set $I$ is a \emph{MIS} of $G.$

\end{enumerate}

\end{definition}

We note that a greedy algorithm need not start with $I,$ in the above definition, being a singleton. In fact, we can use this algorithm to extend an independent set of any size to a maximal independent set. We can also carry out this algorithm on a subgraph of a given graph to extend some independent set into another independent set that is maximal with respect to that subgraph. We are now ready to prove the first result of this section.

\begin{lemma}\label{CONSt}
Let $n \in \mathbb{N}_{\geq 3}$ and $f$ be any well-covered weighting of $S_n.$ Then, $f$ is constant on $V(S_n)-\bigcup_{i=1}^3 C_i.$
\end{lemma}

\begin{proof}
This proof is by induction on $n \in \mathbb{N}_{\geq 3}.$ Choose a pair of adjacent vertices $v_1,v_2 \in V(S_3)-\bigcup_{i=1}^3 C_i.$ Let $T$ be a set that contains vertices adjacent to each vertex in the neighborhood of $v_1$ and of $v_2,$ excluding $v_1$ and $v_2.$ 
\begin{center}
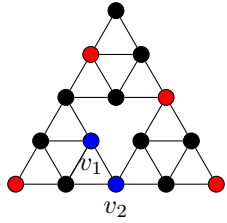
\begin{figure}[H]
\begin{tikzpicture}[x=.15in, y=.15in]
\vertex (w1) at (0,0)[label=below:$$, style={fill=red}]{};
\vertex (w2) at (7,0)[label=below:$$, style={fill=red}]{};
\vertex (w3) at (3.5,6.06)[label=below:$$, style={fill=black}]{};
\vertex (w4) at (3.5,0)[label=below:$v_2$, style={fill=blue}]{};
\vertex (w5) at (1.75,3.03)[label=below:$$, style={fill=black}]{};
\vertex (w6) at (5.25,3.03)[label=below:$$, style={fill=red}]{};
\vertex (w7) at (1.75,0)[label=below:$$, style={fill=black}]{};
\vertex (w8) at (5.25,0)[label=below:$$, style={fill=black}]{};
\vertex (w9) at (0.87,1.51)[label=below:$$, style={fill=black}]{};
\vertex (w10) at (6.13,1.51)[label=below:$$, style={fill=black}]{};
\vertex (w11) at (2.63,1.51)[label=below:$v_1$, style={fill=blue}]{};
\vertex (w12) at (4.37,1.51)[label=below:$$, style={fill=black}]{};
\vertex (w13) at (3.5,3.03)[label=below:$$, style={fill=black}]{};
\vertex (w14) at (2.62,4.53)[label=below:$$, style={fill=red}]{};
\vertex (w15) at (4.38,4.53)[label=below:$$, style={fill=black}]{};
\path 
(w1) edge (w7)
(w4) edge (w7)
(w4) edge (w8)
(w8) edge (w2)
(w1) edge (w9)
(w9) edge (w7)
(w11) edge (w7)
(w11) edge (w4)
(w11) edge (w7)
(w4) edge (w12)
(w8) edge (w12)
(w8) edge (w10)
(w2) edge (w10)
(w9) edge (w11)
(w12) edge (w10)
(w9) edge (w5)
(w5) edge (w11)
(w12) edge (w6)
(w6) edge (w10)
(w5) edge (w13)
(w13) edge (w6)
(w5) edge (w14)
(w14) edge (w13)
(w14) edge (w15)
(w15) edge (w13)
(w6) edge (w15)
(w14) edge (w3)
(w3) edge (w15);
\end{tikzpicture}
\caption{A \emph{MIS} $T$ in red.}
\end{figure}
\end{center}

Since $T \cup \{v_1\}$ and $T \cup \{v_2\}$ are \emph{MIS}s of $S_3,$ by Lemma \ref{RMKKBIR}, $f(v_1)=f(v_2),$ for any  well-covered weighting $f$ of $S_3.$ Repeating this process for each of the remaining five pairs of adjacent vertices in $ V(S_3)-\bigcup_{i=1}^3 C_i,$ we obtain that $f$ is constant on this set.

Assume that the result holds for all $m<n.$ Pick any $S_{n-1}$ subgraph of $S_n.$ By induction, all vertices not in a corner of this $S_{n-1}$ subgraph have the same weight. Choose one out of the two $S_{n-1}$ corners that are not simplicial cliques of $S_n.$ Let $t$ be any vertex in this corner and $t'$ be a vertex adjacent to $t$ that lies outside of this corner. Note that the vertex $t'$ could belong to the chosen $S_{n-1}$ subgraph or to a corner of another $S_{n-1}$ subgraph in the vecinity. This is illustrated in Figure \ref{tprime}. Let $\mathcal{T} \subset V(S_n)- N[\{t,t'\}]$ be an independent set that contains vertices adjacent to each vertex in $N[\{t,t'\}]-\{t,t'\}.$ Extend $\mathcal{T}$ with vertices from $V(S_n)- N[\{t,t'\}]$ via a greedy algorithm. Make this set as large as possible and call it $\mathcal{I}.$

\begin{center}
\begin{figure}[H]
{\begin{tikzpicture}[x=.26in, y=.26in]
\vertex (w1) at (0,0)[label=below:$$, style={fill=black}]{};
\vertex (w2) at (7,0)[label=below:$$, style={fill=red}]{};
\vertex (w3) at (3.5,6.06)[label=below:$$, style={fill=red}]{};
\vertex (w4) at (3.5,0)[label=below:$t$, style={fill=blue}]{};
\vertex (w5) at (1.75,3.03)[label=below:$$, style={fill=red}]{};
\vertex (w6) at (5.25,3.03)[label=below:$$, style={fill=red}]{};
\vertex (w7) at (1.75,0)[label=below:$$, style={fill=black}]{};
\vertex (w8) at (5.25,0)[label=below:$$, style={fill=red}]{};
\vertex (w9) at (0.87,1.51)[label=below:$$, style={fill=red}]{};
\vertex (w10) at (6.13,1.51)[label=below:$$, style={fill=red}]{};
\vertex (w11) at (2.63,1.51)[label=right:$$, style={fill=red}]{};
\vertex (w12) at (4.37,1.51)[label=below:$$, style={fill=red}]{};
\vertex (w13) at (2.62,4.53)[label=below:$$, style={fill=red}]{};
\vertex (w14) at (4.38,4.53)[label=below:$$, style={fill=red}]{};
\vertex (w15) at (3.49,3.03)[label=below:$$, style={fill=red}]{};
\vertex (w16) at (2.62,3.03)[label=below:$$, style={fill=black}]{};
\vertex (w17) at (4.38,3.03)[label=below:$$, style={fill=black}]{};
\vertex (w18) at (.87, 0)[label=below:$$, style={fill=red}]{};
\vertex (w19) at (2.62,0)[label=below:$t'$, style={fill=blue}]{};
\vertex (w20) at (4.38,0)[label=below:$$, style={fill=black}]{};
\vertex (w21) at (6.12,0)[label=below:$$, style={fill=black}]{};
\vertex (w22) at (0.435,0.753)[label=below:$$, style={fill=black}]{};
\vertex (w23) at (1.305,0.753)[label=below:$$, style={fill=black}]{};
\vertex (w24) at (2.175,0.753)[label=below:$$, style={fill=black}]{};
\vertex (w25) at (3.045,0.753)[label=below:$$, style={fill=black}]{};
\vertex (w26) at (3.915,0.753)[label=below:$$, style={fill=black}]{};
\vertex (w27) at (4.785,0.753)[label=below:$$, style={fill=black}]{};
\vertex (w28) at (5.655,0.753)[label=below:$$, style={fill=black}]{};
\vertex (w29) at (6.525,0.753)[label=below:$$, style={fill=black}]{};
\vertex (w30) at (1.75,1.51)[label=below:$$, style={fill=black}]{};
\vertex (w31) at (5.25,1.51)[label=below:$$, style={fill=black}]{};
\vertex (w32) at (1.305,2.26)[label=below:$$, style={fill=black}]{};
\vertex (w33) at (2.175,2.26)[label=right:$$, style={fill=black}]{};
\vertex (w34) at (4.805,2.26)[label=below:$$, style={fill=black}]{};
\vertex (w35) at (5.675,2.26)[label=below:$$, style={fill=black}]{};
\vertex (w36) at (2.185,3.78)[label=below:$$, style={fill=black}]{};
\vertex (w37) at (3.055,3.78)[label=below:$$, style={fill=black}]{};
\vertex (w38) at (3.925,3.78)[label=below:$$, style={fill=black}]{};
\vertex (w39) at (4.795,3.78)[label=below:$$, style={fill=black}]{};
\vertex (w40) at (3.49,4.53)[label=below:$$, style={fill=black}]{};
\vertex (w41) at (3.055,5.28)[label=below:$$, style={fill=black}]{};
\vertex (w42) at (3.925,5.28)[label=below:$$, style={fill=black}]{};
\path
(w1) edge (w18)
(w18) edge (w7)
(w7) edge (w19)
(w19) edge (w4)
(w4) edge (w20)
(w20) edge (w8)
(w8) edge (w21)
(w21) edge (w2)
(w1) edge (w22)
(w22) edge (w23)
(w23) edge (w7)
(w7) edge (w24)
(w24) edge (w19)
(w19) edge (w25)
(w25) edge (w4)
(w4) edge (w26)
(w26) edge (w20)
(w20) edge (w27)
(w27) edge (w8)
(w28) edge (w8)
(w28) edge (w21)
(w21) edge (w29)
(w29) edge (w2)
(w18) edge (w22)
(w18) edge (w23)
(w22) edge (w9)
(w23) edge (w9)
(w9) edge (w32)
(w32) edge (w5)
(w5) edge (w36)
(w36) edge (w13)
(w13) edge (w36)
(w13) edge (w41)
(w41) edge (w3)
(w3) edge (w42)
(w14) edge (w42)
(w14) edge (w39)
(w39) edge (w6)
(w6) edge (w35)
(w35) edge (w10)
(w10) edge (w29)
(w32) edge (w30)
(w9) edge (w30)
(w11) edge (w30)
(w11) edge (w24)
(w30) edge (w33)
(w33) edge (w11)
(w32) edge (w33)
(w5) edge (w33)
(w11) edge (w25)
(w24) edge (w25)
(w26) edge (w12)
(w12) edge (w34)
(w34) edge (w6)
(w34) edge (w35)
(w12) edge (w27)
(w26) edge (w27)
(w12) edge (w31)
(w31) edge (w10)
(w10) edge (w28)
(w28) edge (w29)
(w31) edge (w34)
(w31) edge (w35)
(w6) edge (w17)
(w17) edge (w15)
(w15) edge (w16)
(w16) edge (w5)
(w36) edge (w16)
(w16) edge (w37)
(w37) edge (w15)
(w36) edge (w37)
(w13) edge (w37)
(w13) edge (w40)
(w40) edge (w14)
(w14) edge (w38)
(w38) edge (w15)
(w38) edge (w17)
(w17) edge (w39)
(w38) edge (w39)
(w41) edge (w40)
(w41) edge (w42)
(w40) edge (w42);
\end{tikzpicture}}\hspace{.2in}{\begin{tikzpicture}[x=.26in, y=.26in]
\vertex (w1) at (0,0)[label=below:$$, style={fill=black}]{};
\vertex (w2) at (7,0)[label=below:$$, style={fill=black}]{};
\vertex (w3) at (3.5,6.06)[label=below:$$, style={fill=red}]{};
\vertex (w4) at (3.5,0)[label=below:$$, style={fill=black}]{};
\vertex (w5) at (1.75,3.03)[label=below:$$, style={fill=red}]{};
\vertex (w6) at (5.25,3.03)[label=below:$$, style={fill=black}]{};
\vertex (w7) at (1.75,0)[label=below:$$, style={fill=black}]{};
\vertex (w8) at (5.25,0)[label=below:$$, style={fill=black}]{};
\vertex (w9) at (0.87,1.51)[label=below:$$, style={fill=red}]{};
\vertex (w10) at (6.13,1.51)[label=below:$$, style={fill=red}]{};
\vertex (w11) at (2.63,1.51)[label=right:$$, style={fill=red}]{};
\vertex (w12) at (4.37,1.51)[label=below:$$, style={fill=black}]{};
\vertex (w13) at (2.62,4.53)[label=below:$$, style={fill=red}]{};
\vertex (w14) at (4.38,4.53)[label=below:$$, style={fill=red}]{};
\vertex (w15) at (3.49,3.03)[label=below:$$, style={fill=red}]{};
\vertex (w16) at (2.62,3.03)[label=below:$$, style={fill=black}]{};
\vertex (w17) at (4.38,3.03)[label=below:$$, style={fill=black}]{};
\vertex (w18) at (.87, 0)[label=below:$$, style={fill=red}]{};
\vertex (w19) at (2.62,0)[label=below:$$, style={fill=red}]{};
\vertex (w20) at (4.38,0)[label=below:$t$, style={fill=blue}]{};
\vertex (w21) at (6.12,0)[label=below:$$, style={fill=red}]{};
\vertex (w22) at (0.435,0.753)[label=below:$$, style={fill=black}]{};
\vertex (w23) at (1.305,0.753)[label=below:$$, style={fill=black}]{};
\vertex (w24) at (2.175,0.753)[label=below:$$, style={fill=black}]{};
\vertex (w25) at (3.045,0.753)[label=below:$$, style={fill=black}]{};
\vertex (w26) at (3.915,0.753)[label=below:$$, style={fill=black}]{};
\vertex (w27) at (4.785,0.753)[label=right:$t'$, style={fill=blue}]{};
\vertex (w28) at (5.655,0.753)[label=below:$$, style={fill=black}]{};
\vertex (w29) at (6.525,0.753)[label=below:$$, style={fill=black}]{};
\vertex (w30) at (1.75,1.51)[label=below:$$, style={fill=black}]{};
\vertex (w31) at (5.25,1.51)[label=below:$$, style={fill=black}]{};
\vertex (w32) at (1.305,2.26)[label=below:$$, style={fill=black}]{};
\vertex (w33) at (2.175,2.26)[label=right:$$, style={fill=black}]{};
\vertex (w34) at (4.805,2.26)[label=below:$$, style={fill=red}]{};
\vertex (w35) at (5.675,2.26)[label=below:$$, style={fill=black}]{};
\vertex (w36) at (2.185,3.78)[label=below:$$, style={fill=black}]{};
\vertex (w37) at (3.055,3.78)[label=below:$$, style={fill=black}]{};
\vertex (w38) at (3.925,3.78)[label=below:$$, style={fill=black}]{};
\vertex (w39) at (4.795,3.78)[label=below:$$, style={fill=black}]{};
\vertex (w40) at (3.49,4.53)[label=below:$$, style={fill=black}]{};
\vertex (w41) at (3.055,5.28)[label=below:$$, style={fill=black}]{};
\vertex (w42) at (3.925,5.28)[label=below:$$, style={fill=black}]{};
\path
(w1) edge (w18)
(w18) edge (w7)
(w7) edge (w19)
(w19) edge (w4)
(w4) edge (w20)
(w20) edge (w8)
(w8) edge (w21)
(w21) edge (w2)
(w1) edge (w22)
(w22) edge (w23)
(w23) edge (w7)
(w7) edge (w24)
(w24) edge (w19)
(w19) edge (w25)
(w25) edge (w4)
(w4) edge (w26)
(w26) edge (w20)
(w20) edge (w27)
(w27) edge (w8)
(w28) edge (w8)
(w28) edge (w21)
(w21) edge (w29)
(w29) edge (w2)
(w18) edge (w22)
(w18) edge (w23)
(w22) edge (w9)
(w23) edge (w9)
(w9) edge (w32)
(w32) edge (w5)
(w5) edge (w36)
(w36) edge (w13)
(w13) edge (w36)
(w13) edge (w41)
(w41) edge (w3)
(w3) edge (w42)
(w14) edge (w42)
(w14) edge (w39)
(w39) edge (w6)
(w6) edge (w35)
(w35) edge (w10)
(w10) edge (w29)
(w32) edge (w30)
(w9) edge (w30)
(w11) edge (w30)
(w11) edge (w24)
(w30) edge (w33)
(w33) edge (w11)
(w32) edge (w33)
(w5) edge (w33)
(w11) edge (w25)
(w24) edge (w25)
(w26) edge (w12)
(w12) edge (w34)
(w34) edge (w6)
(w34) edge (w35)
(w12) edge (w27)
(w26) edge (w27)
(w12) edge (w31)
(w31) edge (w10)
(w10) edge (w28)
(w28) edge (w29)
(w31) edge (w34)
(w31) edge (w35)
(w6) edge (w17)
(w17) edge (w15)
(w15) edge (w16)
(w16) edge (w5)
(w36) edge (w16)
(w16) edge (w37)
(w37) edge (w15)
(w36) edge (w37)
(w13) edge (w37)
(w13) edge (w40)
(w40) edge (w14)
(w14) edge (w38)
(w38) edge (w15)
(w38) edge (w17)
(w17) edge (w39)
(w38) edge (w39)
(w41) edge (w40)
(w41) edge (w42)
(w40) edge (w42);
\end{tikzpicture}}
\caption{Examples of $t,$ $t'$ and $\mathcal{I}.$}
\label{tprime}
\end{figure}
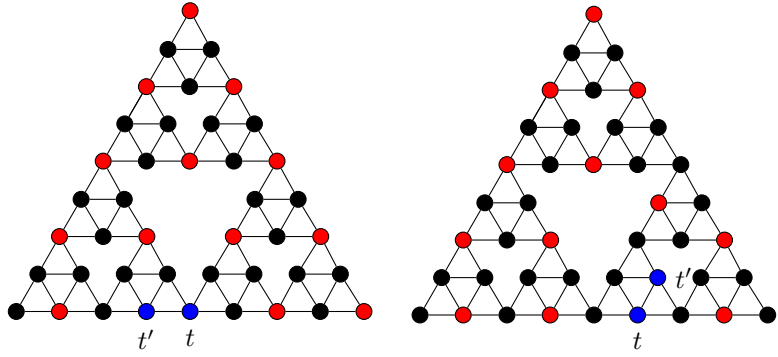
\end{center}

We have that $\mathcal{I} \cup \{t\}$ and $\mathcal{I} \cup \{t'\}$ are independent sets. Now, if $w \in V(S_n)-N[\{t,t'\}],$ $w$ is adjacent to some vertex in $\mathcal{I}.$ If $w \in N[\{t,t'\}]-\{t,t'\},$ $w$ is adjacent to either a vertex in $\mathcal{I},$ to $t$ or $t'.$ Thus,  for any $w \in V(S_n)-\{t,t'\},$ $(\mathcal{I} \cup \{t\}) \cup \{w\}$ and $(\mathcal{I} \cup \{t'\}) \cup \{w\}$ are dependent sets. Hence, $\mathcal{I} \cup \{t\}$ and $\mathcal{I} \cup \{t'\}$ are \emph{MIS}s of $S_n$ so that by Lemma \ref{RMKKBIR}, $f(t)=f(t'),$ for any well-covered weighting $f$ of $S_n.$ 

Since the $S_{n-1}$ subgraph of $S_n$ was chosen without loss of generality, it follows that any well-covered weighting of $S_n$ is constant on the set of all vertices that do not belong to a simplicial clique of $S_n.$
\end{proof}

We now further examine the properties of weightings in the well-covered space of $S_n,$ for $n \in \mathbb{N}_{\geq 3}.$ This result and the corollary that follows it, give us a full description of the well-covered weightings of $S_n.$

\begin{lemma}\label{ZEROO}
Let $n \in \mathbb{N}_{\geq 3}$ and $f$ be a well-covered weighting of $S_n.$ Then, $f(v)=0,$ for all $v$ on the sides of $S_n,$ unless $v$ belongs to a simplicial clique of $S_n,$  where $f$ is constant.
\end{lemma}

\begin{proof}
Let $P_n$ be a side of $S_n$ and $\mathcal{K} \subset V(S_n)-N[P_n]$ be an independent set that contains vertices adjacent to each $v \in N[P_n]-P_n.$ 

\begin{center}
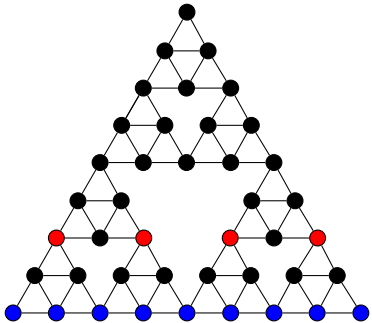
\begin{figure}[H]
\begin{tikzpicture}[x=.26in, y=.26in]
\vertex (w1) at (0,0)[label=below:$$, style={fill=blue}]{};
\vertex (w2) at (7,0)[label=below:$$, style={fill=blue}]{};
\vertex (w3) at (3.5,6.06)[label=below:$$, style={fill=black}]{};
\vertex (w4) at (3.5,0)[label=below:$$, style={fill=blue}]{};
\vertex (w5) at (1.75,3.03)[label=below:$$, style={fill=black}]{};
\vertex (w6) at (5.25,3.03)[label=below:$$, style={fill=black}]{};
\vertex (w7) at (1.75,0)[label=below:$$, style={fill=blue}]{};
\vertex (w8) at (5.25,0)[label=below:$$, style={fill=blue}]{};
\vertex (w9) at (0.87,1.51)[label=below:$$, style={fill=red}]{};
\vertex (w10) at (6.13,1.51)[label=below:$$, style={fill=red}]{};
\vertex (w11) at (2.63,1.51)[label=right:$$, style={fill=red}]{};
\vertex (w12) at (4.37,1.51)[label=below:$$, style={fill=red}]{};
\vertex (w13) at (2.62,4.53)[label=below:$$, style={fill=black}]{};
\vertex (w14) at (4.38,4.53)[label=below:$$, style={fill=black}]{};
\vertex (w15) at (3.49,3.03)[label=below:$$, style={fill=black}]{};
\vertex (w16) at (2.62,3.03)[label=below:$$, style={fill=black}]{};
\vertex (w17) at (4.38,3.03)[label=below:$$, style={fill=black}]{};
\vertex (w18) at (.87, 0)[label=below:$$, style={fill=blue}]{};
\vertex (w19) at (2.62,0)[label=below:$$, style={fill=blue}]{};
\vertex (w20) at (4.38,0)[label=below:$$, style={fill=blue}]{};
\vertex (w21) at (6.12,0)[label=below:$$, style={fill=blue}]{};
\vertex (w22) at (0.435,0.753)[label=below:$$, style={fill=black}]{};
\vertex (w23) at (1.305,0.753)[label=below:$$, style={fill=black}]{};
\vertex (w24) at (2.175,0.753)[label=below:$$, style={fill=black}]{};
\vertex (w25) at (3.045,0.753)[label=below:$$, style={fill=black}]{};
\vertex (w26) at (3.915,0.753)[label=below:$$, style={fill=black}]{};
\vertex (w27) at (4.785,0.753)[label=below:$$, style={fill=black}]{};
\vertex (w28) at (5.655,0.753)[label=below:$$, style={fill=black}]{};
\vertex (w29) at (6.525,0.753)[label=below:$$, style={fill=black}]{};
\vertex (w30) at (1.75,1.51)[label=below:$$, style={fill=black}]{};
\vertex (w31) at (5.25,1.51)[label=below:$$, style={fill=black}]{};
\vertex (w32) at (1.305,2.26)[label=below:$$, style={fill=black}]{};
\vertex (w33) at (2.175,2.26)[label=right:$$, style={fill=black}]{};
\vertex (w34) at (4.805,2.26)[label=below:$$, style={fill=black}]{};
\vertex (w35) at (5.675,2.26)[label=below:$$, style={fill=black}]{};
\vertex (w36) at (2.185,3.78)[label=below:$$, style={fill=black}]{};
\vertex (w37) at (3.055,3.78)[label=below:$$, style={fill=black}]{};
\vertex (w38) at (3.925,3.78)[label=below:$$, style={fill=black}]{};
\vertex (w39) at (4.795,3.78)[label=below:$$, style={fill=black}]{};
\vertex (w40) at (3.49,4.53)[label=below:$$, style={fill=black}]{};
\vertex (w41) at (3.055,5.28)[label=below:$$, style={fill=black}]{};
\vertex (w42) at (3.925,5.28)[label=below:$$, style={fill=black}]{};
\path
(w1) edge (w18)
(w18) edge (w7)
(w7) edge (w19)
(w19) edge (w4)
(w4) edge (w20)
(w20) edge (w8)
(w8) edge (w21)
(w21) edge (w2)
(w1) edge (w22)
(w22) edge (w23)
(w23) edge (w7)
(w7) edge (w24)
(w24) edge (w19)
(w19) edge (w25)
(w25) edge (w4)
(w4) edge (w26)
(w26) edge (w20)
(w20) edge (w27)
(w27) edge (w8)
(w28) edge (w8)
(w28) edge (w21)
(w21) edge (w29)
(w29) edge (w2)
(w18) edge (w22)
(w18) edge (w23)
(w22) edge (w9)
(w23) edge (w9)
(w9) edge (w32)
(w32) edge (w5)
(w5) edge (w36)
(w36) edge (w13)
(w13) edge (w36)
(w13) edge (w41)
(w41) edge (w3)
(w3) edge (w42)
(w14) edge (w42)
(w14) edge (w39)
(w39) edge (w6)
(w6) edge (w35)
(w35) edge (w10)
(w10) edge (w29)
(w32) edge (w30)
(w9) edge (w30)
(w11) edge (w30)
(w11) edge (w24)
(w30) edge (w33)
(w33) edge (w11)
(w32) edge (w33)
(w5) edge (w33)
(w11) edge (w25)
(w24) edge (w25)
(w26) edge (w12)
(w12) edge (w34)
(w34) edge (w6)
(w34) edge (w35)
(w12) edge (w27)
(w26) edge (w27)
(w12) edge (w31)
(w31) edge (w10)
(w10) edge (w28)
(w28) edge (w29)
(w31) edge (w34)
(w31) edge (w35)
(w6) edge (w17)
(w17) edge (w15)
(w15) edge (w16)
(w16) edge (w5)
(w36) edge (w16)
(w16) edge (w37)
(w37) edge (w15)
(w36) edge (w37)
(w13) edge (w37)
(w13) edge (w40)
(w40) edge (w14)
(w14) edge (w38)
(w38) edge (w15)
(w38) edge (w17)
(w17) edge (w39)
(w38) edge (w39)
(w41) edge (w40)
(w41) edge (w42)
(w40) edge (w42);
\end{tikzpicture}
\caption{An example of $\mathcal{K}$ (red) for $S_4.$}
\end{figure}
\end{center}

Via a greedy algorithm, we can extend $\mathcal{K}$ with vertices from $\mathcal{K} \subset V(S_n)-N[P_n]$ into an independent set $\mathcal{K}'$ that is as large as possible. Note that for any independent set $\mathcal{M} \subset P_n$ that is maximal with respect to $P_n,$ $\mathcal{K}' \cup \mathcal{M}$ is a \emph{MIS} of $S_n.$ Thus, the vertices of $P_n$ behave as if they were the vertices of an isolated path. By Thoerem \ref{BIRNAL}, since $|P_n| \geq 5,$ the two vertices at each end of $P_n$ have the same weight, while all other vertices of $P_n$ have weight zero, under any well-covered weighting of $S_n.$
\end{proof}

\begin{corollary}
For any well-covered weighting $f$ of $S_n$ with $n \in \mathbb{N}_{\geq 3},$ $f(v)=0,$ for all $v \in V(S_n)-\bigcup_{i=1}^3 C_i.$ 
\end{corollary}

\begin{proof}
Let $P_n$ be a side of $S_n$ and suppose that $v \in V(P_n)$ but that $v \notin C_i,$ for all $i \in \mathbb{N}_{\leq 3}.$ Let $f$ be any well-covered weighting of $S_n.$ By Lemma \ref{ZEROO}, $f(v)=0,$ and by Lemma \ref{CONSt}, it follows that $f$ is zero on $V(S_n)-\bigcup_{i=1}^3 C_i.$
\end{proof}

In conclusion, if $n \in \mathbb{N}_{\geq 3},$ any well-covered weighting of $S_n$ is a linear combination of at most three linearly independent functions that assign a distinct, non-zero scalar to each simplicial clique, and zero to all other vertices of $S_n.$ Then, $wcdim(S_n) \leq 3$ and by Remark \ref{LBOUND}, we obtain the final theorem for this section.   

\begin{theorem}
For any $n \in \mathbb{N}_{\geq 3},$ $wcdim(S_n)=3.$
\end{theorem}

\section*{acknowledgements}
The author gratefully acknowledges support from NSF Grant \#DMS-1156273, the California State University Fresno Mathematics REU program, especially Oscar Vega and Tam\'as Forg\'acs, the Mellon Mays Undergraduate Fellowship, and the editor and referee for their kind suggestions. 


\end{document}